\documentclass[11pt,a4paper]{article}
\usepackage{amssymb,amsfonts,amsthm,amsmath}%Packages usuels
\usepackage{dsfont}%Differentielle
\usepackage{mathrsfs}%Police curviligne

\theoremstyle{definition}
\newtheorem{Def}{Definition}[section]
\theoremstyle{plain}
\newtheorem{Pro}[Def]{Proposition}
\newtheorem{Lem}[Def]{Lemma}
\newtheorem{The}[Def]{Theorem}
\newtheorem{Cor}[Def]{Corollary}

\theoremstyle{remark}

\newtheorem{Rem}[Def]{Remark}
\renewcommand\[{\begin{equation}}
\renewcommand\]{\end{equation}}
\newcommand{\NN}{\mathbb{N}}
\newcommand{\ZZ}{\mathbb{Z}}
\newcommand\RR{\mathbb{R}}

\newcommand\WW{\mathbb{W}}
\newcommand{\defeq}{\mathrel{\mathop:}=}

\newcommand\di[1]{\mathrm{d}#1}
\newcommand\ind{\mathds{1}}
\newcommand{\norme}[1]{\left\Vert #1\right\Vert}
\newcommand\abs[1]{| #1 |}
\newcommand{\est}{\lesssim}%estimation

\newcommand{\cL}{\mathcal{L}}
\usepackage[utf8]{inputenc}
\usepackage[english]{babel}
\usepackage[T1]{fontenc}
\usepackage{graphicx}
\usepackage{accents}%Barre du bas
\usepackage[refpage,intoc]{nomencl}%Nomenclature

\makenomenclature
\usepackage{enumerate}

\usepackage{hyperref}
\hypersetup{                    % parametrage des hyperliens
    colorlinks=true,                % colorise les liens
    breaklinks=true,                % permet les retours à la ligne pour les liens trop longs
    urlcolor= blue,                 % couleur des hyperliens
    linkcolor= red,                % couleur des liens internes aux documents (index, figures, tableaux, equations,...)
    citecolor= blue                % couleur des liens vers les references bibliographiques
    }
\usepackage{fancyhdr}
\newcommand\tY{\widetilde{Y}}
\newcommand\ty{\widetilde{y}}

\newcommand{\bunderline}[1]{\mkern2mu\underline{\mkern-2mu#1\mkern-4mu}\mkern4mu }%bar du bas
\newcommand\ab[1]{\bunderline{#1}}%Elements du model space
\newcommand\atau{\ab{\tau}}
\newcommand\af{\ab{f}}
\newcommand\aX{\ab{X}}
\newcommand\aY{\ab{Y}}

\newcommand\atY{\widetilde{\ab{Y}}}
\newcommand\aZ{\widetilde{\ab{Z}}}
\newcommand\aDelta{\ab{\Delta}}
\newcommand\aA{\ab{A}}
\newcommand\aR{\ab{R}}
\newcommand\an{\ab{1}}
\newcommand\aW{\ab{W}}
\newcommand\aWd{\dot{\ab{W}}}
\newcommand\rW{\textbf{W}}

\newcommand\aWWd{\ab{\dot{\WW}}}
\newcommand\aWaWd{\ab{W\dot{W}}}
\newcommand\model{D^\gamma_M}
\newcommand\rmodel{D^{2\alpha}_{M^r}}
\newcommand\rinter{(1/3,1/2]}
\newcommand\ahol{C^\alpha}
\newcommand\aI{I}
\newcommand{\Rc}{{\cal{R}}}
\newcommand\FF{\hat{F}}
\newcommand\DD{{\mathscr{D}}}
\newcommand\rC{\mathscr{C}}
\newcommand{\rnorme}[1]{\norme{#1}_{{2\alpha},T}^*}
\newcommand{\rsnorme}[1]{\norme{#1}_{{2\alpha},T}}
\newcommand\mul{\star}

\newcommand\ga{\Gamma_{t,s}}
\newcommand{\proj}{Q_{<2\alpha}}

\newcommand\modspace{{\mathcal{T}}}
\newcommand\indexSet{{\mathcal{A}}}
\newcommand\grou{\mathcal{G}}

\newcommand\TT{{\mathscr{T}}}
\newcommand\vect[1]{\text{Vect}\langle #1 \rangle}
\newcommand\vv{\text{Vect}}

\title{Solving rough differential equations with the theory of regularity structures}
\author{Antoine Brault
\footnote{Institut de Mathématiques de Toulouse. CNRS UMR 5219.
Université Paul Sabatier
118, route de Narbonne, F-31062 Toulouse Cedex 09, France.
Email: \texttt{antoine.brault@math.univ-toulouse.fr}}}

\begin{document}
\maketitle

\begin{abstract}
The purpose of this article is to solve rough differential equations with the theory of  regularity structures. These new tools recently developed by Martin Hairer for solving semi-linear partial differential stochastic equations were inspired by the rough path theory. We take a pedagogical approach to facilitate the understanding of this new theory. We recover results of the rough path theory with the regularity structure framework. Hence, we show how to formulate a fixed point problem in the abstract space of modelled distributions to solve the rough differential equations. We also give  a proof of the existence of a rough path lift with  the theory of regularity structure.
\end{abstract}
\tableofcontents
%%%%%%%%%%%%%%%%%%%%%%%%%%%%%%%%%%%%%%%%%%%%%%%%%%%%%%%%%%%%%%%%%%%%%%%%
\section{Introduction}
Let $T>0$ be a finite time horizon.
Suppose that we want to solve the following ordinary differential equation
\[\forall t\in [0,T],~\di{y_t} = F(y_t)\di{W_t},\qquad y_0=\xi,
\label{rde}\]
where $W: [0,T] \rightarrow \RR^n$ and $F:\RR^d \rightarrow \mathcal{L}(\RR^{n},\RR^d)$ are regular functions. The equation \eqref{rde} can be reformulated as
\[\forall t\in [0,T],\quad y_t = \xi + \int_{0}^{t}F(y_u)\di{W_u}\label{int_rde}.\]  
When $W$ is smooth, the equation \eqref{rde} is well-defined as $$y_t = \xi + \int_{0}^{t}F(y_u)\dot{W}_u\di{u},$$
where $\dot{W}$ represents the derivative of $W$.
 Therefore, it becomes an ordinary differential equation that can be solved by a fixed-point argument.

Unfortunately,  there are many natural situations in which we would like to consider the equation of type \eqref{int_rde} for an irregular path $W$. This is notably the case when dealing with stochastic processes. For example the paths of the Brownian motion are almost surely nowhere differentiable \cite{karatzas}. It is then impossible to interpret \eqref{rde} in a classical sense. Indeed, even if $\dot{W}$ is understood as a distribution, it is not possible in general to define a natural product between distributions, as $y$ is itself to be thought as a distribution.

On the one hand, to overcome this issue, Itô's theory \cite{karatzas} was built to define properly an integral against a martingale $M$ (for example the Brownian motion) : $\int_0^t Z_u \di{M}_u$, where $Z$ must have some good properties. The definition is not pathwise as it involves a limit in probability. Moreover, this theory is successful to develop a stochastic calculus with martingales but fails when this property vanishes. This is the case for the fractional Brownian motion, a natural process in modelling. Another bad property is that the map $W\mapsto y$ is not continuous in general with the associated uniform topology \cite{lyons91}.

On the other hand, L.C. Young proved in \cite{young} that we can define the integral of $f$ against $g$ if $f$ is $\alpha$-Hölder, $g$ is $\beta$-Hölder with $\alpha +\beta >1$ as 

\begin{align*}
\int_{0}^t f\di{g} =\lim_{|P|\rightarrow 0} \sum_  {u,v \text{ successive points in }P} f(u)(g(v)-g(u)),
\end{align*}
 where $P$ is a subdivision of $[0,t]$ and $|P|$ denotes its mesh. This result is sharp, so that it is not possible to extend it to the case $\alpha + \beta \leq 1$ \cite{young}. If $W$ is $\alpha$-Hölder it seems natural to think that $y$ is $\alpha$-Hölder, too. So assuming $\alpha<1/2$ then $2\alpha<1$, and  Young's integral fails to give a meaning to \eqref{rde}. The fractional Brownian motion which depends on a parameter $H$ giving its Hölder regularity cannot be dealt with  Young's integral as soon as $H\leq 1/2$.

T. Lyons introduced in \cite{lyons} the rough path
theory which overcomes  Young's limitation. The main idea is to construct for $0\leq s\leq t\leq T$  an object $\WW_{s,t}$ which \textquotedblleft looks 
like\textquotedblright\ $\int_{s}^{t}(W_u-W_s)\di{W}_u$ and then define an integral against $(W,\WW)$. This is done with the sewing lemma (Theorem~\ref{classic_rough_int}).
This theory enabled to solve \eqref{rde} in most of the cases and to define a topology such that the Itô map $(W,\WW)\mapsto y$ is continuous.
Here, the rough path $(W,\WW)$ \textquotedblleft encodes\textquotedblright\ the path $W$ with algebraic operations. It is an extension of the Chen series developed in \cite{chen} and \cite{lyons94} to solve controlled differential equations.
Since the original article of T. Lyons, other approaches of the rough  paths theory were developed in \cite{davie}, \cite{gubinelli} and  \cite{bailleul}.
The article \cite{coutinlejay} deals with the linear rough equations with a bounded operators.
For monographs about the rough path theory, the reader  can refer to \cite{qian} or \cite{friz_victoir}.

Recently, M. Hairer developed in \cite{hairer} the theory of regularity structures which can be viewed as a generalisation of the rough path theory. It allows to give a good meaning and to solve singular stochastic partial differential equations (SPDE). One of the main ideas is to build solutions from  approximations at different scales. This is done with the  reconstruction theorem (Theorem~\ref{the_reconstruction_map}).
 Another fruitful  theory was introduced to solve SPDE in \cite{gip} and also studied in 
\cite{bailleulbernicotfrey}.

The main goal of this article is to make  this new theory  understandable to people who are  familiar with rough differential equations or ordinary differential equations.

Thus, we propose to solve \eqref{rde} with the theory of regularity structures, when the Hölder regularity of $W$ is in $(1/3,1/2]$. In particular, we build the rough integral (Theorem \ref{classic_rough_int}) and the tensor of order $2$: $\WW$ (Theorem \ref{levyArea}) with the reconstruction theorem.

Our approach is very related to \cite[Chapter~$13$]{frizhairer}
where is established the link between rough differential equations and the theory of regularity structures. However, we give here the detailed 
 proofs of Theorem~\ref{levyArea} and Theorem~\ref{classic_rough_int} with the reconstruction theorem. It seems important to make the link between the two theories but is skipped in \cite{frizhairer}.

This article can be read without knowing about rough path or regularity structure theories.

  After introducing notation in Section~\ref{sec:notation},  we introduce in  Section \ref{sec_holder} the Hölder spaces which allow us to \textquotedblleft measure\textquotedblright\ the regularity of a function. Then, we present the rough path theory in  Section~\ref{sec_elements_rough_path_theory}. In the Sections \ref{sec_regularitystructures} and \ref{sec_modelledDistribution} we give the framework of the theory of  regularity structures and the modelled distributions for solving \eqref{rde}. We prove in Sections \ref{sec_roughInt} and \ref{sub_rough_lift} the existence of the controlled rough path integral and the existence of a rough path lift. Finally, after having defined the composition of a function with a modelled distribution in Section \ref{sec_comp}, we solve the rough differential equation \eqref{rde} in  Section \ref{sec_solve}. 

%%%%%%%%%%%%%%%%%%%%%%%%%%%%%%%%%%%%%%%%%%%%%%%%%%%%%%%%%%%%%%%%%%%
\section{Notations}
\label{sec:notation}
We denote by $\cL(A,B)$\nomenclature{$\mathcal{L}(A,B) $}{Linear continuous maps between two vector spaces $A$ and $B$}, the set of linear continuous maps between two vector spaces $A$ and $B$.
Throughout the article, $C$ denotes a  positive constant whose value may change. For two functions $f$ and $g$, we write $f\est g$\nomenclature{$\est $}{Estimation} if there is a constant $C$ such that $f\leq C g$. The symbol $\defeq$\nomenclature{$\defeq$}{Symbol to define an object} means that the right hand side of the equality defines the left hand side. For a function $Z$ from $[0,T]$ to a vector space, its increment is denoted by $Z_{s,t}\defeq Z_t-Z_s$\nomenclature{$Z_{s,t}$}{Means $Z_t-Z_s$}. If $X_1,...,X_k$ are $k$  vectors of a linear space, we denote by $\vect{X_1,...,X_k}$\nomenclature{$\vect{\cdot}$}{Subspace generated by the linear combinations of $\cdot$.} the subspace generated by the linear combinations of these vectors. Let $T$ be a non-negative real, we denote by $[0,T]$ a compact interval of $\RR$. For a continuous function $f: [0,T]\rightarrow E$, where $\left(E,\norme{\cdot}\right)$ is a Banach space, we denote by $\norme{f}_{\infty,T}$\nomenclature{$\norme{f}_{\infty,T}$}{Supremum of $\norme{f(t)}$ for  $t\in [0,T]$} the supremum of $\norme{f(t)}$ for  $t\in [0,T]$. The tensor product is denoted by $\otimes$\nomenclature{$\otimes$}{Tensor product}.
We denote $\lfloor\cdot\rfloor$\nomenclature{$\lfloor\cdot\rfloor$}{Floor function} the floor function.
%%%%%%%%%%%%%%%%%%%%%%%%%%%%%%%%%%%%%%%%%%%%%%%%%%%%%%%%%%%%%%%%%%%%%%%%%
\section{Hölder spaces}
\label{sec_holder}

\subsection{Classical Hölder spaces with a positive exponent}
We introduce  Hölder spaces which allow us to characterize the regularity of a non-differentiable function.
\begin{Def}
\label{holder1}
For $0\leq \alpha < 1$ and $T>0$, the function $f:[0,T]\rightarrow E$ is $\alpha$-Hölder if
$$\sup_{s\ne t\in [0,T]} \frac{\norme{f(t)-f(s)}}{\abs{t-s}^\alpha}< + \infty.$$
We denote by $C^\alpha(E)$\nomenclature{$\ahol$}{$\alpha$-Hölder continuous functions} the \emph{space of $\alpha$-Hölder functions} equipped with the semi-norm $$\norme{f}_{\alpha,T} \defeq \sup_{s\ne t\in [0,T]} \frac{\norme{f(t)-f(s)}}{\abs{t-s}^\alpha}.$$

If $\alpha \geq 1$ such that $\alpha = q + \beta$ where $q\in \NN$ and $\beta\in [0,1)$, we set $f\in C^\alpha(E)$ if $f$ has $q$ derivatives and $f^{(q)}$ is $\beta$-Hölder, where $f^{(q)}$ denotes the derivative of order $k$ ($f^{(0)}\defeq f$).
%Hölder. Furthermore, we define $\ahol_b$\nomenclature{$\ahol_b$}{Hölder space of bounded functions} the space such that $f^{(q)}$ is bounded by a constant independent of $T$. 

We denote by $\ahol=\ahol(\RR^n)$. For $q\in \NN$, we denote by $C^q_b$\nomenclature{$C^q_b$}{Bounded functions with all derivative up to order $q$ bounded. } the set of all functions $f\in C^q$ such that
\[\norme{f}_{C^q_b}\defeq \sum_{k=0}^q\norme{f^{(k)}}_\infty<+\infty.\]

Finally, for $q\in\NN$, we define $C_0^q$ the set of functions in $C^q_b$ with a compact support.
\end{Def}

\begin{Rem}
The linear space of $\alpha$-Hölder functions $\ahol(E)$ is a non separable Banach space endowed with one of the two equivalent norms $\norme{f(0)}+\norme{f}_{\alpha,T}$ or $\norme{f}_{\infty,T}+\norme{f}_{\alpha,T}$.
\end{Rem}

%\begin{Rem}
%If $\alpha>1$ and $\norme{f}_{\alpha,T}<+\infty$, then $f$ is a constant. Indeed, $$\frac{\norme{f(t)-f(s)}}{\abs{t-s}}\leq \norme{f}_{\alpha,T} \abs{t-s}^{\alpha-1},$$
%we conclude that $f'(t)=0$ for all $t\in ]0,T[.$
%\end{Rem}

\begin{Rem}
If $f$ is $\alpha$-Hölder on $[0,T]$, then $f$ is $\beta$-Hölder for $\beta<\alpha$, \emph{i.e.} $C^\alpha(E) \subset C^\beta(E)$.
\end{Rem}

%Later, we extend the Definition \ref{holder1} to the negative orders in the Definition~\ref{holder2}.

\subsection{Localised test functions and Hölder spaces with a negative exponent}

In equation~\eqref{rde}, typically $W$ is in $\ahol$ with $\alpha\in (0,1)$. We need to deal with the derivative of $W$ is the sense of distribution which should be of negative Hölder regularity $\alpha-1<0$. We give in this section the definition of the space $\ahol$ with $\alpha<0$. We show in Lemma~\ref{lem_lien_distribution_fonction} that an Hölder function is
$\alpha$-Hölder if and only if the derivative in the sense of distribution is $\alpha-1$-Hölder with $\alpha\in (0,1)$.

%In the next section, we define a model of a regularity structure. A model maps~$\modspace$ with the space of Schwartz distributions ${\cal{D}}'(\RR,\RR^n)$, \emph{\emph{i.e.}}, the space of linear continuous forms  from ${\cal{D}}(\RR,\RR^n)$ (the space of smooth functions compactly supported from $\RR$ to $\RR^n$ with the topology of uniform convergence over all compact intervals). 
For $r>0$, we denote by $B_r$\nomenclature{$B_r$}{The space of all functions in $\phi\in C^r_b$ compactly supported on $[-1,1]$, such that 
$\norme{\eta}_{C^r_b}\leq 1$.}  the space of all functions in $\eta\in C^r_b$ compactly supported on $[-1,1]$, such that 
$\norme{\eta}_{C^r_b}\leq 1$.

%
%We define Hölder spaces with negative exponent which is a space of distributions.

\begin{Def}
\label{def:test_functions}
For $\lambda >0,~s\in\RR$ and a test function $\eta\in B_r,$ we define the  \emph{test function localised} at $s$ by
$$\eta_s^\lambda(t) \defeq \frac{1}{\lambda}\eta\left(\frac{t-s}{\lambda}\right),$$
for all $t\in \RR$.
\end{Def}

\begin{Rem}
The lower  is $\lambda$, the more $\eta_s^\lambda$ is localised at $s$, as can be seen in Figure \ref{localised}.
\begin{figure}[!ht]
\centering
\includegraphics[scale=0.7]{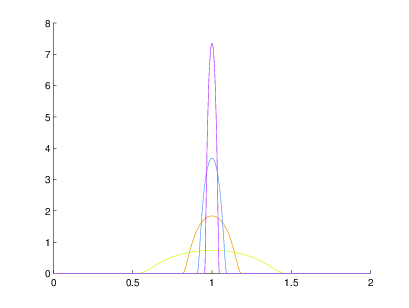}
\caption{Representation of $\eta_s^\lambda$ for $s=1$, $\lambda\in\{0.5,0.2,0.1,0.05\}$ and with $\eta(s) = \exp(-1/(1-s^2))\ind_{(-1,1)}(s).$}
\label{localised}
\end{figure}
\end{Rem}

\begin{Rem}
We work here with $t,s\in \RR$, because we want to solve stochastic ordinary differential equations. But in the case of  stochastic partial differential equations, the parameters $t$ and $s$ belong to $\RR^e$ where $e$ is an integer, see \cite{hairer}.
\end{Rem}

\begin{Def}
\label{holder2}
For $\alpha<0$, we define the \emph{Hölder space} $C^\alpha$ as elements in the dual of $C^r_0$ where $r$ is an integer strictly greater than $-\alpha$ and such that for any $\xi\in C^\alpha$ the following estimate holds
\[|\xi(\eta_s^\lambda)| \leq C(T) \lambda^\alpha,\label{eq:negative_holder}\]
 where $C(T)\geq 0$ is a constant uniform over all $s\in [0,T]$, $\lambda\in  (0,1]$ and  $\eta\in B_r$.

We define the semi-norm on $C^\alpha$ as the lowest constant $C(T)$ for a fixed compact $[0,T]$, i.e 
\begin{align*}
{\xi}_{\alpha,T}\defeq\sup_{s\in [0,T]}\sup_{\eta\in B_r}\sup_{\lambda\in (0,1]} \left|\frac{\xi(\eta_s^\lambda)}{\lambda^\alpha}\right|.
\end{align*}
\end{Def}

\begin{Rem}
The space $C^\alpha$ does not depend on the choice of $r$, see for example \cite{frizhairer} Exercise 13.31, p. 209.
\end{Rem}

\begin{Rem}
With Definition \ref{holder1}, we can give a meaning of an $\alpha$-Hölder function for $\alpha\in \RR$.
Moreover it is possible to show that if $f$ is a function in $C^\alpha$ with $\alpha = q+\beta>0$ where $q$ is an integer and $\beta\in(0,1)$, then for every $x\in [0,T]$ and localised functions $\eta_x^\lambda$, 
$$\abs{(f-P_x)(\eta_x^\lambda)}\leq C \lambda^\beta,$$
where $C$ is uniform over $x\in[0,T]$, $\lambda\in (0,1]$ and $\eta\in B_r$ ($r$ a positive integer),
$P_x$ is the Taylor expansion of $f$ of the order $q$ in $x$, and $f-P_x$ is view as the canonical function associated.

Now, when we say that $f\in C^\alpha$ we should distinguish two cases :
\begin{itemize}
\item if $\alpha\geq 0$, $f$ is an $\alpha$-Hölder \emph{function} in the sense of Definition \ref{holder1}
\item if $\alpha<0$, $f$ is an $\alpha$-Hölder \emph{distribution} in the sense of  Definition \ref{holder2}.
\end{itemize}
\end{Rem}

We give here a characterization of the space $C^\alpha$ for $\alpha\in (-1,0)$ which is useful to make a link between the rough path and the regularity structures theories.
\begin{Lem}
\label{lem_lien_distribution_fonction}
For any $\beta\in (0,1)$, the distribution $\xi\in C^{\beta-1}$ if and only if there exist a  function $z\in C^\beta$ such that $z(0)=0$ and
\[\forall \eta\in C^1_0,~ \xi(\eta)=-\langle z,\eta'\rangle.\]
Which means that $z'=\xi$ in the sense of distribution.
Moreover, for all $t\in [0,1]$,
\begin{align*}
z(t) = \sum_{k\in I_l}\langle\xi,\phi^l_k\rangle\int_0^t\phi^l_k+\sum_{j\geq l}\sum_{k\in I_j} \langle\xi,\psi^j_k\rangle\int_0^t\psi^j_k,
\end{align*}
where $\phi,\psi$ are defined in Theorem~\ref{def:wavelet} with a compact support in $[-c,c]$ ($c\geq 0$), $l$ is an integer such that $2^{-l}c\leq 1$ and $I_j:=[-\lfloor c\rfloor,2^j+\lfloor c\rfloor]\bigcap\ZZ$.
\end{Lem}
The proof of  Lemma \ref{lem_lien_distribution_fonction} requires to introduce elements of the wavelet theory. The proof of the following theorem can be found in \cite{meyer}.
\begin{The}
\label{def:wavelet}
There exist $\phi,\psi\in C^1_0(\RR)$ such that  for all $n\in \NN$
\[\{\phi_k^i\defeq 2^{i/2}\phi(2^i\cdot-k),~ k\in \ZZ\}\cup\{\psi_k^j\defeq 2^{j/2}\psi(2^{j}\cdot-k),~k\in \ZZ,~j\geq i\}\]
is an orthonormal basis of $L^2(\RR)$.
This means that for all $f\in L^2(\RR)$, $i\in\NN$ we can write
\[f(t)=\sum_{j\geq i}^{+\infty}\sum_{k\in \ZZ}\langle f,\psi_k^j\rangle \psi^j_k(t)+\sum_{k\in \ZZ}\langle f,\phi^i_k\rangle\phi^i_k(t),\label{eq:decomp_ondelettes}\]
where the convergence is in $L^2(\RR)$. Moreover, we have the very useful property,
\[\int \psi(t)t^k\di{t}=0,\]
for $k\in \{0,1\}$. 
% Moreover we can choose $\phi,\psi$ such that  $f\in C^\beta$ with $\beta\in (0,1)$ if and only if the coefficients of the decomposition
%\eqref{eq:decomp_ondelettes} have the bounds
%\[\forall s\in \ZZ,~ |<f,\psi^n_s>|\leq C_1 2^{-\frac{n}{2}}2^{-n\beta},~\mathrm{and} ~\forall s\in\ZZ,~|<f,\phi^0_s>|\leq C_2,  \]
%where $C_1,C_2\geq 0$ are two constants and $n\in \NN$.
\end{The}
\begin{Rem}
The notation in  Definition~\ref{def:test_functions} for $\eta_s^\lambda$ and in  Theorem~\ref{def:wavelet} for $\phi_k^i$, $\psi_k^j$ are similar but the meaning
are slightly different.
\end{Rem}

We now proceed to the proof of Lemma \ref{lem_lien_distribution_fonction}.
\begin{proof}[Proof of Lemma~\ref{lem_lien_distribution_fonction}]
The first implication is trivial and does not require the wavelet analysis.
If there exists $z\in C^\alpha$ such that for any $\eta\in C^1_0$,
$\xi(\eta)=-\langle z,\dot{\eta}\rangle$, then for $\lambda\in (0,1)$ and $s\in\RR$, 
\begin{align*}
\xi(\eta_s^\lambda)&=-\frac{1}{\lambda^2}\int_\RR z(u)\dot{\eta}\left(\frac{u-s}{\lambda}\right)\di{u}\\
&=-\frac{1}{\lambda^2}\int_\RR (z(u)-z(-\lambda+s))\dot{\eta}\left(\frac{u-s}{\lambda}\right)\di{u},
\end{align*}
where the last equality holds because $\eta$ is compactly supported.

With the condition $\eta\in B_1$, $u\mapsto \dot{\eta}((u-s)/\lambda)$
is supported on\\ $[-\lambda+s,\lambda+s]$, which yields to the bound
\[|\xi(\eta_s^\lambda)|\leq 2\norme{\eta}_{C^1}\norme{z}_\alpha\lambda^{\alpha-1},\]
and proves that $\xi\in C^{\alpha-1}$.

Now, we prove the converse. Let $\phi,\psi\in C^1_0$ be defined in Theorem~\ref{def:wavelet}. Let $c\geq 0$ be a constant such that supports of
$\phi$ and $\psi$ are in $[-c,c]$. We denote $l$ an integer such that
$2^{-l}c\leq 1$. Thus, the support of $\phi_0^l$ is in $[-1,1]$ and
the support of $\psi_k^j$ is smaller than $2$ for $j\geq l$.

For $\xi\in C^{\alpha-1}$ for $\alpha\in (0,1)$ we  define for $t\in [0,1]$,
\[z(t)\defeq  \sum_{k\in\ZZ}\langle\xi,\phi^l_k\rangle\int_0^t\phi^l_k+\sum_{j\geq l}\sum_{k\in\ZZ} \langle\xi,\psi^j_k\rangle\int_0^t\psi^j_k\label{eq:def_z}.\]
Noting that for $j\geq l$ and $k\in\ZZ$, $\phi_k^l$ and $\psi_k^l$ are compactly
supported in $[2^{-j}(k-c),2^{-j}(k+c)]$, the terms $\int_0^1\phi_k^j$ and 
$\int_0^1\psi_k^j$ vanish when
$2^{-j}(k+c)\leq 0$ and $1\leq 2^{-j}(k-c)$.
 Thus, we can rewrite \eqref{eq:def_z} as
 \begin{align}
 z(t) = \sum_{k\in I_l}\langle\xi,\phi^l_k\rangle\int_0^t\phi^l_k+\sum_{j\geq l}\sum_{k\in I_j} \langle\xi,\psi^j_k\rangle\int_0^t\psi^j_k,\label{eq:def_z_2} 
 \end{align}
 where $I_j:=[-\lfloor c\rfloor,2^j+\lfloor c\rfloor]\bigcap\ZZ$.
%We show that \eqref{eq:def_z_2} is defined respect to the Hölder norm
%and that \mbox{$z\in C^\alpha$}. 
%We can assume that $\phi,\psi$ are compactly supported in $[0,1]$, thus the functions $\psi_k^i$ are compactly supported in $[2^{-j}k,2^{-j}(k+1)]$. We  rewrite \eqref{eq:def_z} as
%\[z\defeq \langle\xi,\phi\rangle\int_0^t\phi+\sum_{j\in \NN}\sum_{k=0}^{2^j-1} \langle\xi,\psi^j_k\rangle\int_0^t\psi^j_k\label{eq:def_z_2},\]
%because the terms $\langle\psi_k^j,\ind_{[0,1]}\rangle$ vanish for $k\notin [0,2^j-1]$.
The series on the right hand side of \eqref{eq:def_z} converges in the sense of distributions.
We need to justify that the limit $z$ is in $C^\alpha$.

We denote for any integer $N\in\NN$,
\[S_N^z\defeq \sum_{j=l}^NS_j,\]
where $S_j(t)\defeq \sum_{k\in I_l} \langle\xi,\psi^j_k\rangle\int_0^t\psi^j_k$.
According to \eqref{eq:negative_holder}, for all $j\geq l$ and $k\in I_j$
\begin{align*}
|\langle\xi,\psi_k^j\rangle|\leq C 2^{j/2-j\alpha}.
\end{align*}
%Thus, we compute for $|t-s|\leq 1$,
%\begin{align*}
%\sum_j=0^N|S_j(t)-S_j(s)|&\leq C\norme{\phi}_\infty |t-s|+\sum_{j\in \NN}|S_j(t)-S_j(s)|. \label{eq:z_t}
%\end{align*} 
For $|t-s|\leq 1$, let $j_0\leq N$ be an integer such that 
$2^{-j_0}\leq |t-s|<2^{-j_0+1}$. This is always possible for $N$ large enough. On the one hand, if $l\leq j_0$, for $l\leq j\leq j_0$,
\begin{align}
\nonumber
|S_j(t)-S_j(s)|&\leq \norme{S'_j}_\infty|t-s|\\
\nonumber
&\leq |t-s|\sup_{u\in [0,1]}\sum_{I_j}|\langle\xi,\psi_k^j\rangle|\cdot|\psi_k^j(u)|\\
&\leq C2^{j(1-\alpha)}|t-s|,\label{eq:S_j0}
\end{align}
where we use the fact that $\sum_{k\in I_j}|\psi(2^jt-k)|\leq C$ for a constant $C\geq0$, because $\psi$ is compactly supported.
On the other hand, for $j>\max{(j_0,l)}$,
\begin{align}
|S_j(t)-S_j(s)|& \leq 2 \norme{S_j}_\infty\\
&\leq 2 \sup_{u\in [0,1]}\sum_{k=0}^{2^{j}-1}|\langle\xi,\psi_k^j\rangle|\cdot |2^{-j}\hat{\psi}_k^j(u)|,\\
&\leq 2C2^{-j\alpha}\sup_{u\in [0,1]}\sum_{k=0}^{2^j-1}|\hat{\psi}(2^ju-k)|,\label{eq:S_j1}
\end{align}
where $\hat{\psi}\defeq \int_0^t\psi$. Because $\int_\RR \psi=0$, there is a constant $C'\geq 0$ independent of $j$ such that $\sum_{k=0}^{2^j-1}|\hat{\psi}(2^ju-k)-\hat{\psi}(-k)|<C'$.
So finally, for $j>\max{(j_0,l)}$,
\[|S_j(t)-S_j(s)|\leq C 2^{-j\alpha}\label{eq:S_j2}.\] 
Thus, combining \eqref{eq:S_j0}, \eqref{eq:S_j2}, for $N\geq l$,
\begin{align*}
\sum_{j=l}^N|S_j(t)-S_j(s)|&\leq C|t-s|\sum_{j=l}^{j_0}2^{j(1-\alpha)}+C\sum_{j=j_0+1}^\infty 2^{-j\alpha}\\
&\leq C' |t-s|^\alpha,
\end{align*}
where $C'$ is a new constant ($\sum_{j=l}^{j_0}2^{j(1-\alpha)}=0$ if $j_0<l$).
It follows that $\norme{S^z_N}_{\alpha,1}$ is uniformly bounded in $N$ and thus that $z\in C^\alpha$.
% It follows that $\sum_{j=l}^N\norme{S_j}_{\alpha,1}$ converges. This proves that $S_N^z$ converges  in $C^\alpha$ and justifies the definition of $z$ in
%\eqref{eq:def_z} and \eqref{eq:def_z_2}.

%Moreover the convergence of \eqref{eq:def_z} is uniform according to Theorem \ref{the:wavelet_regularity}.
Now, we want to check that $\xi=\dot{z}$ in the distribution framework. For any $\eta\in C^2_0$,
\begin{align*}
\langle z,\dot{\eta}\rangle &=\sum_{k\in \ZZ} \langle\xi,\phi^l_k\rangle\Big\langle\int_0^t\phi^l_k,\dot{\eta}\Big\rangle+\sum_{j\geq l}\sum_{k\in \ZZ} \langle\xi,\psi^j_k\rangle\Big\langle\int_0^t\psi^j_k,\dot{\eta}\Big\rangle\\
&=-\sum_{k\in \ZZ} \langle\xi,\phi^l_k\rangle\langle\phi^l_k,\eta\rangle-\sum_{j\geq  l}\sum_{k\in \ZZ}\langle\xi,\psi^j_k\rangle\langle\psi^j_k,\eta\rangle\\
&=-\langle\xi,\sum_{k\in \ZZ} \phi^l_k\langle\phi^l_k,\eta\rangle+\sum_{j\geq l}\sum_{k\in \ZZ}\psi^j_k\langle\psi^j_k,\eta\rangle\rangle\\
&= -\langle\xi,\eta\rangle,
\end{align*}
where the commuting of the serie and $\xi$ is justified by the continuity of $\xi$ in $C^1_0$ and the convergence of the following serie in $C^1_0$,
\[S_N^\eta\defeq\sum_{j=l}^N\sum_{k\in\ZZ}\psi^j_k\langle\psi^j_k,\eta\rangle.\]
Indeed, we have

\begin{align*}
|\langle\psi_k^j,\eta\rangle|&\leq 2^{-j/2}\left|\int \psi(x)\eta(2^{-j}(x+k))\di{x}\right|\\
&\leq 2^{-j/2}\int |\psi(x)||\eta(2^{-j}(x+k))-\eta(2^{-j}k)-\eta'(2^{-j}k)(2^jx)|\di{x}\\
&\leq \norme{\psi}_\infty\norme{\eta'}_\infty 2^{-j/2}2^{-2j},
\end{align*}
where we use the fact that $\int\psi(t)t^k=0$ for an integer $k\leq 1$. This implies that
\begin{align*}
\sum_{j=l}^N\norme{\sum_{k\in\ZZ}\psi^j_k\langle\psi^j_k,\eta\rangle}_{C^1}&\leq \sum_{j=l}^N\norme{\sum_{k\in\ZZ}\psi^j_k\langle\psi^j_k,\eta\rangle}_\infty +\sum_{j=l}^N\norme{\sum_{k\in\ZZ}2^{j}\psi'^j_k\langle\psi^j_k,\eta\rangle}_\infty\\
&\leq \left(\sum_{j=0}^N2^{-2j}+\sum_ {j=0}^N2^{-j}\right )\norme{\psi}_\infty\norme{\eta'}_\infty,
\end{align*}
which proves that $S^\eta_N$ is absolutely convergent in $C^1_0$.

Now by density of $C^2_0$ in $C^1_0$ and the continuity of $\xi$ on $C^1_0$ we conclude that $\langle z,\dot{\eta}\rangle=-\langle\xi,\eta\rangle$ holds for $\eta\in C^1_0$.
\end{proof}

%%%%% ROUGH PATH THEORY %%%%%%%%%%%%%%%%%%%%%
\section{Elements of  rough path theory}
We introduce here the elements of the rough path theory for solving Equation~\eqref{int_rde}. The notions discussed are reformulated in the regularity structure framework in the following sections.  For an extensive introduction the reader can refer to \cite{frizhairer}, and for complete monographs to \cite{qian, friz_victoir}.
\label{sec_elements_rough_path_theory}

%%%%% THE SPACE OF ROUGH PATHS%%%%%%%%%%%%%%%%%%%%%%
\subsection{The space of rough paths}
\label{subsec_roughpath}
Let $W$ be a continuous function from $[0,T]$ to $\RR^n$.

We set $\alpha\in (1/3,1/2]$. 
Then, \eqref{int_rde} has not meaning, because the integral term is not defined.  The main idea of the rough path theory is to define an object $\WW_{s,t}$ which has the same algebraic and analytical properties as $\int_{s}^{t} W_{s,u}\otimes \di{W_u}$, the integral of the increment of the path against itself.

The importance of the iterated integrals can be understood with the classical linear differential equations where the solutions are provided with the exponential function. Indeed, if $W:[0,T]\rightarrow \RR$ is smooth, the solutions of \[\di{y_t}=y_t\di{W_t}\]
are
\[y_{s,t}=\exp(W_{s,t})=1+\int_s^t\di{W_{t_1,s}}+\int_s^t\int_s^{t_1}
\di{W_{t_2,s}}\di{W_{t_1,s}}+\dotsb.\]

\begin{Def}
\label{def_rough_path}
An \emph{$\alpha$-Hölder rough path} with $\alpha\in (1/3,1/2]$ is an ordered pair $\rW \defeq (W,\WW)$\nomenclature{$\rW$}{Rough path associated to $W$} of functions, where $W:[0,T]\rightarrow \RR^n$ and $\WW:[0,T]^2 \rightarrow \RR^n\otimes \RR^n$  such that
\begin{enumerate}
\item For $s,u,t\in [0,T]$, $\WW_{s,t} - \WW_{s,u} -\WW_{u,t}= W_{s,u}\otimes W_{u,t}$ (Chen's relation), i.e., for every $1\leq i,j\leq n$,
$\WW_{s,t}^{i,j} - \WW_{s,u}^{i,j} -\WW_{u,t}^{i,j}= W_{s,u}^iW_{u,t}^j$. 
\item The function $W$ is $\alpha$-Hölder and  $\WW$ is $2\alpha$-Hölder in the sense
$$\norme{\WW}_{2\alpha,T} \defeq \sup_{s\ne t \in [0,T]} \frac{\norme{\WW_{s,t}}}{\abs{t-s}^{2\alpha}} < + \infty.$$
\end{enumerate}
One calls $\WW$ the \emph{second order process}.
We denote by $\rC^\alpha$ the \emph{space of $\alpha$-Hölder rough paths} endowed with the semi-norm 
$$\norme{\rW}_{\alpha,T} = \norme{W}_{\alpha,T} + \norme{\WW}_{2\alpha,T}.$$
\end{Def}

\begin{Rem}
The second order process $\WW_{s,t}$ can be thought of as $\int_{s}^{t} W_{s,u}\otimes \di{W_u}$.
\end{Rem}

\begin{Rem}
The first condition which is called Chen's relation represents the algebraic property of $\int_{s}^{t} W_{s,u}\otimes \di{W_u}$. Indeed, if $W$ is smooth,
 $$\int_{s}^t W_{s,v}^i\dot{W}_v^j\di{v} - \int_{s}^u W_{s,v}^i\dot{W}_v^j\di{v} - \int_{u}^t W_{u,v}^i\dot{W}_v^j\di{v} = W_{s,u}^i W_{u,t}^j$$
for all $1\leq i,j\leq n$ and $0\leq s\leq u\leq t$.
\end{Rem}

\begin{Rem}
The second condition is also an extension of the analytic property of the smooth case.
\end{Rem}

\begin{Rem}
\label{uniq}
If $\WW$ is a second order process of $W$, for any $2\alpha$-Hölder function $F$ taking values in $\RR^n\otimes \RR^n$, $(s,t)\mapsto\WW_{s,t} + F_t-F_s$
 satisfies also the two properties of Definition \ref{def_rough_path}.
So if $\WW$ exists, it is not unique at all.
\end{Rem}

Building $\WW$ from $W$ is  non-trivial  as soon as $n\geq 2$.
\begin{The}
\label{levyArea}
For any $W\in C^\alpha$ with $\alpha\in (1/3,1/2]$ there exists a rough path lift $\WW$, \emph{i.e.} $\rW=(W,\WW)\in \rC^\alpha$ in a way that the map $W\mapsto \rW$ is continuous for the topology defined in Definition~\ref{def_rough_path}.
\end{The}

\begin{proof}
This result was proved in \cite{victoir}.
We prove of this result in the case $\alpha\in (1/3,1/2]$ in  Section~\ref{sub_rough_lift} as an application of the reconstruction theorem (Theorem~\ref{the_reconstruction_map}).
\end{proof}

%%%%% CONTROLLED ROUGH PATHS%%%%%%%%%%%%%%%%%
\subsection{Controlled rough paths}
\label{subsec_controlledroughpaths}
The aim of this section is to define an integrand against $\rW$, called a controlled rough path  by $W$. This approach was developed by M.~Gubinelli in \cite{gubinelli}. We introduce a function with the same regularity as $W$ which is not differentiable with respect to time but with respect to $W$ itself. This is the concept of the Gubinelli's derivative.

\begin{Def}
\label{contrough}
Let $W$ be in $C^\alpha$, we call a \emph{controlled rough path  by $W$} the pair $(y,y')\in C^\alpha(\RR^d)\times \ahol(\RR^{d\times n})$ such that
\[y_{s,t}=y'_sW_{s,t}+R_{s,t}^y\label{gub},\]
with $\norme{R^y}_{2\alpha,T}<+\infty.$ The function  $y'$ is the \emph{Gubinelli's derivative} of $y$ with respect to $W$.

We denote $\DD^{2\alpha}_W$\nomenclature{$\DD^{2\alpha}_W$}{The controlled rough path space} the \emph{space of the controlled rough paths} $(y,y')$ driven by $W$ endowed with the semi-norm

\[\norme{(y,y')}_{2\alpha,T}^W \defeq \norme{y'}_{\alpha,T}+\norme{R^y}_{2\alpha,T}\label{norme_cont_rough}.\]
\nomenclature{$\norme{\cdot}_{2\alpha,T}^W$}{Semi-norm of the controlled rough path space}
\end{Def}

\begin{Rem}
The identity \eqref{gub} looks like a Taylor expansion of first order
$$f_t = f_s + f'_s(t-s) + O(|t-s|^2),$$
but $(W_t-W_s)$ substitutes the usual polynomial expression $(t-s)$, $y'_s$ the normal derivative and the remainder term is of order $2\alpha$ whereas order $2$. The theory of regularity structures is a deep generalization of this analogy.
\end{Rem}

\begin{Rem}
The Gubinelli's derivative $y'$ is matrix-valued which depends on $y$ and $W$.
\end{Rem}

\begin{Rem}
Unlike the rough path space $\rC^\alpha$ (see Definition \ref{def_rough_path}) which is not a linear space, $\DD^{2\alpha}_W$ is a Banach space with the norm $\norme{y_0}+\norme{y'_0}+\norme{(y,y')}_{2\alpha,T}^W$ or the norm $\norme{y}_{\infty,T}+\norme{y'}_{\infty,T}+\norme{(y,y')}_{2\alpha,T}^W$. These two norms are equivalent.
\end{Rem}

\begin{Rem}
The uniqueness  of $y'$ depends on the regularity of $W$.
If $W$ is too smooth, for example in $C^{2\alpha}$, then $y$ is in $C^{2\alpha}$, and every continuous function $y'$ matches with the definition of the Gubinelli's derivative, particularly $y'=0$. But we can prove that $y'$ is uniquely determined by $y$ when $W$ is  irregular enough.
The reader can refer to the Chapter 4 of \cite{frizhairer} for detailed explanations.
\end{Rem}

%\begin{Rem}
%%It is not trivial to exhibit a non trivial controlled rough path. 
%In application, functions in the space $\DD^{2\alpha}_W$ match with solutions of rough differential equations.
%\end{Rem}
%

%%%%%%% INTEGRATION AGAINST ROUGH PATH %%%%%%%%%%%%%%%%%%%%%%%
\subsection{Integration against rough paths}
\label{subsec_integration}
If $F$ is a linear operator $A$, the differential equation \eqref{rde} can be restated on an integral form as \[y_t = \xi + A\int_0^t y_u\di{W_u}\label{intrde}.\]
To give a meaning to \eqref{intrde}  we must define an integral term
$\int_0^ty_u\di{W_u}$.

 When $W\in C^\alpha$, $y\in C^\beta$ with $\alpha + \beta >1$, we are able to define \eqref{intrde} with  Young's integral. Unfortunately, the solution $y$ of \eqref{intrde} inherits of the regularity of $W$. Hence,  Young's theory  allows us to solve \eqref{intrde} only when $\alpha > 1/2$.

When $\alpha\in (1/3,1/2]$, we need to \textquotedblleft improve\textquotedblright\ the path $W$ in taking into account of $\WW$ in  the definition of the integral. 

%We build the rough integral of a controlled rough path $(y,y')\in \DD^{2\alpha}_W$ against a rough path $(W,\WW)\in \rC^{\alpha}$.

\subsection{Young's integration}
Young's integral was developed by Young in \cite{young} and then used by T.~Lyons in \cite{lyons94} to deal with differential equations driven by a stochastic process.

The integral is defined with a Riemann sum. Let $\cal{P}$\nomenclature{${\cal{P}}$}{A subdivision} be a subdivision of $[s,t]$, we denote by $|\cal{P}|$\nomenclature{$|\cale{P}|$}{The mesh of the subdivision $\cal{P}$} the mesh of $\cal{P}$. We want to define the integral as follows:
$$\int_{s}^t y_u\di{W_u} =\lim_{|{\cal{P}}|\to 0} \sum_{u,v\in \cal{P}} y_u W_{u,v},$$
where $u,v\in \cal{P}$ denotes successive points of the subdivision.
\begin{The}
If $W \in C^\alpha$ and $y\in C^\beta$ with $\alpha + \beta > 1$, $\sum_{u,v\in \cal{P}} y_u W_{u,v}$ converges when $|{\cal{P}}|\to 0$. The limit is independent of the choice of $\cal{P}$, and it is denoted as $\int_s^t y_u\di{W_u}$. Moreover the bilinear map $(W,y)\rightarrow \int_s^t y_u\di{W_u}$ is continuous from $C^\alpha\times C^\beta$ to $C^\alpha$.
\end{The}

\begin{proof}
For the original proof cf. \cite{young}.
\end{proof}

Some important properties of the classical Riemann integration holds.
\begin{Pro}

\begin{enumerate}
\item Chasles' relation holds.
\item When $t\to s$ we have the following approximation
\[\int_s^t y_u\di{W_u} = y_sW_{s,t} + O(|t-s|^{\alpha+\beta}). \]
\item The map $t\mapsto \int_s^t y_u\di{W_u}$ is $\alpha$-Hölder continuous.
\item If $F$ is $C^1$, $F(y)$ is $C^\beta$-Hölder and the Young integral $\int_s^tF(y_u)\di{W_u}$ is well-defined as above.
\end{enumerate}
\end{Pro}

\begin{Rem}
\label{remarkDis}
Unfortunately with  Young's construction, when $\alpha \leq 1/2$, we can find two sequences of smooth functions $W^{1,n}$ and $W^{2,n}$ converging to $W$ in $C^\alpha$  but such that $\int_s^t F(W^{1,n})\di{W^{1,n}}$ and $\int_s^t F(W^{2,n})_n\di{W^{2,n}}$ converge to two different limits for a smooth function $F$ . See for an example the Lejay's area bubbles in \cite{lejay}.
\end{Rem}

\subsection{Controlled rough path integration}
The rough integral relies on the controlled rough paths introduced previously.
Remark \ref{remarkDis} shows that if $y,W\in C^\alpha$, we cannot define a continuous integral such as $\int_s^ty_u\di{W_u}$ looks like $y_sW_{s,t}$ when $t\to s$. We must use the structure of controlled rough paths to define a \textquotedblleft good\textquotedblright\ integral of $y$ against $W$. Then, given a rough path $\rW\in \rC^\alpha$ and considering  a controlled rough path $(y,y')\in\DD^{2\alpha}_W$ we would like to build an integral $\int_s^t y_u\di{W_u}$ as a good approximation of $y_sW_{s,t} + y'_s\WW_{s,t}$ when $t\to s$.

\begin{The}
\label{classic_rough_int}
For $\alpha\in (1/3,1/2]$, let $\rW=(W,\WW)\in \rC^{\alpha}$ be an $\alpha$-Hölder rough path. Given a controlled rough path driven by $W$ : $(y,y')\in \DD^{2\alpha}_W$ we consider the sum $\sum_{u,v\in \cal{P}} y_uW_{u,v} + y'_{u}\WW_{u,v}$ where $\cal{P}$ is a subdivision of $[s,t]$ ($s\leq t\in [0,T]$). This sum converges when the mesh of $\cal{P}$  goes to 0. We define the integral of $y$ against  $\rW$ as

$$\int_s^t y_u\di{\rW_u} \defeq \lim_{|{\cal{P}}|\to 0} \sum_{u,v\in \cal{P}} y_uW_{u,v} + y'_{u}\WW_{u,v}.$$

The limit exists and does not depend on the choice of the subdivision.
Moreover, the map $(y,y')\rightarrow (t\in [0,T]\mapsto \int_0^ty_u\di{\rW_u},y)$ from $\DD^{2\alpha}_W$ into itself is continuous.
\end{The}

\begin{proof}
The classical proof uses the sewing lemma \cite[Lemma~$4.2$]{frizhairer}. We give a proof with the reconstruction theorem (Theorem~\ref{the_reconstruction_map}) in Section~\ref{sec_roughInt}.
\end{proof}

To solve \eqref{rde}, we need to show that if $F$ is a smooth function, then $F(y_t)$ remains a controlled rough path. The following proposition shows that $(F(y),(F(y))')$ defined by :
\[F(y)_t=F(y_t),\qquad F(y)_t' = F'(y_t)y_t,\label{eq_comp}\]
is a controlled rough path.
\begin{Pro}
Let $F:\RR^d\rightarrow\cL(\RR^n,\RR^d)$ be a function twice continuously differentiable such that $F$ and its derivative are bounded. Given $(y,y')\in \DD^{2\alpha}_W$ let
 $(F(y),F(y)')\in \DD^{2\alpha}_W$ defined as above \eqref{eq_comp}. Then, there is a constant $C_{\alpha,T}$ depending only on $\alpha$ and $T$ such as
$$\norme{F(y),F(y)'}_{2\alpha,T}^{W}\leq C_{\alpha,T}  \norme{F}_{C^2_b}(1+\norme{W}_{\alpha})^2(\norme{y_0'} + \norme{y,y'}_{2\alpha,T}^W)^2,$$
where $\norme{F}_{C^2_b}= \norme{F}_{\infty}+\norme{F'}_\infty +\norme{F''}_\infty.$
\end{Pro}

\begin{proof}
We can find the proof in \cite{frizhairer}. This proposition is equivalent to  Theorem \ref{def_comp}, which is formulated in the regularity structure framework.
\end{proof}

%%%%%%%%%%%%%%% REGULARITY STRUCTURES %%%%%%%%%%%%%%%%%%%%%%%%%%%%%
\section{Regularity structures}
\label{sec_regularitystructures}
\subsection{Definition of a regularity structure}
The theory of  regularity structures was introduced by Martin Hairer in \cite{hairer}. The tools developed in this theory allow us to solve a very wide range of semi-linear partial differential equations driven by an irregular noise.

This theory can be viewed as a generalisation of the Taylor expansion theory to irregular functions. The main idea is to describe the regularity of a function at small scales and then to  reconstruct this function with the reconstruction operator of Theorem~\ref{the_reconstruction_map}.

First we give the definition of a regularity structure.

\begin{Def}
\label{regstr}
A \emph{regularity structure} is a 3-tuple $\TT=(\indexSet,\modspace,\grou)$\nomenclature{$\TT$}{A regularity structure} where
\begin{itemize}
\item The \emph{index set} $\indexSet\subset\RR$\nomenclature{$\indexSet$}{Index set of a regularity structure} is bounded from below, locally finite and such that $0\in \indexSet$.
\item The \emph{model space} $\modspace$\nomenclature{$\modspace$}{Model space of a regularity structure} is a graded linear space indexed by $\indexSet$ : $\modspace=\bigoplus_{\alpha\in \indexSet} \modspace_\alpha$, where each $\modspace_\alpha$ is a non empty Banach space. The  elements of $\modspace_\alpha$ are said of \emph{homogeneity} $\alpha$.
For $\atau\in \modspace$, we denote $\norme{\atau}_\alpha$\nomenclature{$\norme{\atau}_\alpha$}{Norm of the component of $\atau$ in $\modspace_\alpha$} the norm of the component of $\atau$ in $\modspace_\alpha$. Furthermore, $\modspace_0=\vect{\an}$ is isomorphic to $\RR$.
\item The set $\grou$\nomenclature{$\grou$}{Structure group of a regularity structure} is a set of continuous linear operators acting on $\modspace$ such as for $\Gamma\in \grou,$ $\Gamma(\an)=\an$ and $\atau\in \modspace_\alpha$, $\Gamma\atau - \atau \in \bigoplus_{\beta<\alpha} \modspace_\beta.$ The set $\grou$ is called   
 \emph{structure group}.
\end{itemize}
\end{Def}

\begin{Rem}
We underline the elements of the model space for the sake of clarity.
\end{Rem}

\begin{Rem}
We set $m\defeq\min \indexSet$,  $\Gamma\atau=\atau$ for every $\atau\in\modspace_m$.
\end{Rem}

Let us explain the motivations of this definition. The classic polynomial Taylor expansion of order $m\in \NN$ is given, between  $0$ and $t \in\RR$, where $t$ converges to $0$ by

$$f(t) = P(t)+ o(t^m),\quad \mathrm{where}~P(t) =\sum_{k=0}^m \frac{f^{(k)}(0)}{k!}t^k.$$
In this case the approximation $P$ of $f$ is indexed by integers and the space $\modspace$ is the  polynomial space. For all $h\in \RR$, the operator
$\Gamma_h$ associates a Taylor expansion at point $t$ with a Taylor expansion at a point  $t+h$.
The polynomial $\Gamma_h(P(t))-P(t)$ is of order less than $m-1$ :
$$\Gamma_h(P(t))-P(t) \defeq P(t+h)-P(t) = \sum_{k=0}^m \frac{f^{(k)}(0)}{k!}((t+h)^k-t^k).$$

Moreover we have the structure of group on $(\Gamma_h,h\in \RR)$ :
\begin{eqnarray*}
\Gamma_{h'}\circ\Gamma_h P(t)=\Gamma_{h'}\left(\sum_{k=0}^{d-1}\frac{f^{k}(0)}{k!}(t+h)^k\right)
&= &\sum_{k=0}^{d-1}\frac{f^{k}(0)}{k!}\Gamma_{h'}((t+h)^k)\\
&=&\sum_{k=0}^{d-1}\frac{f^{k}(0)}{k!}(t+h+h')^k\\
& = &\Gamma_{h+h'}P(t).
\end{eqnarray*}

Hence, we can define the polynomial regularity structure as following.
\begin{Def}
We define $\TT^1=(\indexSet^1,\modspace^1,\grou^1)$  the \emph{canonical polynomial regularity structure} as
\begin{itemize}
\item $\indexSet^1=\NN$ is the index set.
\item For $k\in \indexSet^1$ we define $\modspace^1_k=\vect{\ab{X^k}}$. The subspace $\modspace^1_k$ contains the monomial of order $k$. The polynomial model space is $\modspace^1=\bigoplus_{k\in \indexSet}\modspace^1_k$.
\item For $h\in\RR$, $\Gamma^1_h\in \grou^1$ is given by
$$\Gamma^1_h(\ab{X^k})=(\aX+h\an)^{k}.$$

For $P_k\in \modspace^1_k$, there is $a_k\in \RR$ such that $P_k=a_k\ab{X^k}$. We define the norm on $\modspace^1_k$ by $\norme{P_k}_k=|a_k|$.
\end{itemize}
\end{Def}

With the same arguments we define the polynomial regularity structure and its model associated in $\RR^n$.
\begin{Def}
\label{def_poly}
We define $\TT^p=(\indexSet^p,\modspace^p,\grou^p)$  the \emph{canonical polynomial regularity structure} on $\RR^n$  as
\begin{itemize}
\item $\indexSet^p=\NN$ is the index set.
\item For $\delta\in \indexSet^p$, and $k$ a multi-index of $\NN^n$ such that
$|k|\defeq k_1+\dots+k_n=\delta $, we define $T_\delta^p=\vect{{\aX^k}\defeq\prod_{i=1}^n{\aX^{k_i}_i}, |k|=\delta}$. This space $T^p_k$ is a linear space of homogeneous polynomial with $n$ variables and of order $\delta$. For $P_\delta\in T_\delta^p,$ there are real coefficients
$(a_k)_{|k|=\delta}$ such that $P_\delta=\sum_{|k|=\delta}a_k \aX^k.$
We chose the norm on $T_\delta^p$ such that $\norme{P_{\delta}}_\delta\defeq \sum_{|k|=\delta}|a_k| $. 

We define $\modspace^p=\bigoplus_{\delta\in A}T^p_\delta$ as the polynomial model space.
\item For $h\in\RR^n$, $\Gamma^p_h\in G^p$ is given by
$$\Gamma^p_h({\aX^k})=\prod_{i=1}^n({\aX_i}+h_i\an)^{k_i}.$$
\end{itemize}
\end{Def}

\begin{Rem}
The polynomial regularity structure is a trivial example of regularity structure which we introduce for a better understanding. But the strength of this theory is to deal with negative degree of homogeneity. 
\end{Rem}

\subsection{Definition of a model}

\begin{Def}
\label{model}
Given a regularity structure $\TT = (\indexSet,\modspace,\grou)$, a \emph{model} $M=(\Pi,\Gamma)$\nomenclature{$M$}{Model associated to a regularity structure} is two sets of functions such that for any $s,t,u\in \RR$
\begin{itemize}
\item The operator $\Pi_s$ is  continuous and linear from $\modspace$ to ${\cal{D}}'(\RR,\RR^n)$.
\item $\ga$ belongs to $\grou$, so it is a linear operator acting on $\modspace$.
\item The following algebraic relations hold: $\Pi_s\Gamma_{s,t}=\Pi_t$ and $\Gamma_{s,t}\Gamma_{t,u}=\Gamma_{s,u}$.
\item The following analytic relations hold : for every $\gamma >0$, $\beta <\alpha\leq \gamma$ with $\alpha,\beta\in \indexSet$ and $\atau\in \modspace_\alpha$,   there  is a constant $C(T,\gamma)$ uniform over $s,t\in [0,T]$, $\lambda\in (0,1]$,  $\phi\in B_r$ such that
\begin{align}
\nonumber
|\Pi_{s}(\atau)(\phi_{s}^\lambda)|&\leq C(T,\gamma)\lambda^\alpha\norme{\atau}_{\alpha}\\
\mathrm{~and~} \norme{\Gamma_{s,t}(\atau)}_{\beta}&\leq C(T,\gamma) |t-s|^{\alpha-\beta}\norme{\atau}_{\alpha} \label{def_model_ineq}.
\end{align}

We denote respectively by  $\norme{\Pi}_{\gamma,T}$ and $\norme{\Gamma}_{\gamma,T}$  the smallest constants such that the bounds \eqref{def_model_ineq} hold. Namely,
\begin{align*}
&\norme{\Pi}_{\gamma,T} \defeq \sup_{s\in [0,T]}\sup_{\phi\in B_r}\sup_{\lambda\in (0,1]}\sup_{\alpha<\gamma}\sup_{\atau\in T_{\alpha}} \frac{\abs{\Pi_{s}(\atau)(\phi_{s}^\lambda)}_{\alpha}}{\lambda^\alpha\norme{\atau}_{\alpha}}\\
\mathrm{and~~}&\norme{\Gamma}_{\gamma,T} \defeq \sup_{s\ne t \in [0,T]}\sup_{\beta<\alpha<\gamma}\sup_{\atau\in T_\alpha} \frac{\norme{\Gamma_{s,t}(\atau)}_{\beta}}{|t-s|^{\alpha-\beta}\norme{\atau}_{\alpha}}.
\end{align*}
 The two operators $\norme{\cdot}_{\gamma,T}$ define semi-norms.

\end{itemize}
\end{Def}

The easiest regularity structure which we can describe is the polynomial one (see Definition \ref{def_poly}). We can now define the model associated to this regularity structure.

\begin{Def}
Given that $\TT^p=(\indexSet^p,\modspace^p,\grou^p)$\nomenclature{$\TT^p$}{Polynomial regularity structure}  the canonical polynomial regularity structure on $\RR^n$ defined in the Definition \ref{def_poly}, we define the \emph{ model of the polynomial regularity structure} $M^p=(\Pi^p,\Gamma^p)$\nomenclature{$M^p$}{Model of the polynomial regularity structure} such that
 for all $x, y\in \RR^n$ and $k$ a multi-index of order $n$,
 $$\Pi^p_x({\aX^k})(y)\defeq ((y_1-x_1)^{k_1},\dots,(y_n-x_n)^{k_n}),$$
 $$\Gamma^p_{x,y}({\aX^k})\defeq\Gamma_{x-y}({\aX^k}).$$
\end{Def}

\begin{proof}
It is straightforward to check that this definition is in accordance with the one of  a model (Definition~\ref{defMod} below).
\end{proof}

\begin{Rem}
The operator $\Pi_s$ which associates to an element of the abstract space a distribution which approximates this element in $s$.
Typically for polynomial regularity structure on $\RR$,
$$\Pi^p_s(\ab{X^k})=(t\rightarrow (t-s)^k).$$
\end{Rem}

\begin{Rem}
In the model space, the operator $\Gamma_{s,t}$ gives an expansion in a point $s$, given an expansion in a point $t$. For example
\[\label{gamma}\Gamma^p_{s,t}(\ab{X^k})=\Gamma^p_{s-t}(\ab{X^k})=(\ab{X}+(s-t)\an)^k.\]
\end{Rem}

\begin{Rem}
The first algebraic relation means that if a distribution looks like $\atau$ near $t$, the same distribution looks like $\Gamma_{s,t}(\atau)$ near $s$. In practice, we use this relation to find the suitable operator $\ga$.
The second algebraic relation is natural. It says that moving an expansion from $u$ to $s$ is the same as moving an expansion from $u$ to $t$ and then from $t$ to $s$. 
\end{Rem}

\begin{Rem}
The first analytic relation has to be understood as $\Pi_s$ approximating $\atau\in T_\alpha$  in $s$ with the precision $\lambda^\alpha$. The relation \eqref{gamma} shows that the second analytic relation is natural. Indeed,
$$(\ab{X}+(t-s)\an)^k = \sum_{i=0}^k \binom{k}{i}(t-s)^{k-i}\ab{X^i},$$
so for $\ell\leq k$, $\norme{\Gamma^p_{s,t}(\ab{X^k})}_\ell = \binom{k}{i}\abs{t-s}^{k-\ell}$, where $\binom{k}{i}=\frac{k!}{i!(k-i)!}$ are the binomial coefficients.
\end{Rem}

%%%%%% THE ROUGH PATH REGULARITY STRUCTURE %%%%%%%%%%%%%%%%%%%%%%%%%%%%%%%%%%%%%%
\subsection{The rough path regularity structure}
\label{roughstru}
We now reformulate the results of  Subsections \ref{subsec_roughpath} and \ref{subsec_controlledroughpaths} to build up a regularity structure.

In order to find the regularity structure of rough paths, we make some computations for $n=1$. Then, we give the proof in the general case after Definition \ref{roughStruDef}.

We fix $\alpha\in(1/3,1/2]$ and a rough path $\rW=(W,\WW)\in \rC^\alpha$. We show how to build the regularity structure of rough paths.

Let $(y,y')\in\DD^{2\alpha}_W$ be a controlled rough path.
According to  Definition~\ref{contrough},  $y_t=y_s+y'_sW_{s,t}+O(|t-s|^{2\alpha})$.
To describe the expansion of $y$ with the regularity structure framework, we set the symbol $\an$ constant of homogeneity $0$ and the symbol $\aW$ of homogeneity $\alpha$.
This leads us to define the elements of the regularity structure of the controlled rough path $(y,y')$ evaluated at time $t$ by
$$\aY(t) = y_t\an + y'_t\aW.$$

Moreover, we would like to build the rough path integral $\int y\di{\rW}$ in the regularity structure context. So we introduce abstract elements $\aWd$ and $\aWWd$ which \textquotedblleft represent\textquotedblright\ $\di{\rW}=\di{(W,\WW)}$. The function $W$ is $\alpha$-Hölder, so we define the homogeneity of $\aWd$ as $\alpha-1$. The second order process $\WW$ is $2\alpha$-Hölder, which leads us to define the homogeneity of $\aWWd$ as $2\alpha -1$.

Finally, with the notation of  Definition \ref{regstr},
$A = \{\alpha-1, 2\alpha-1,0,\alpha\}$, $\modspace=\vv\langle\aWd,\aWWd,\an,\aW\rangle$. Besides, we order the elements in $\vv\langle\cdot\rangle$ by homogeneity. 

It remains to define $\grou$ and an associated model. We start by building the model $(\Pi,\Gamma)$. For $s\in[0,T]$, $\Pi_s$ should transform the elements of $\modspace$ to distributions (or functions when it is possible) which approximate this elements at the point $s$. On the one hand we define 
$$\Pi_s(\aWd)(\phi) \defeq \int \phi(t) \di{W_t},\qquad \Pi_s(\aWWd)(\phi) \defeq \int \phi(t) \di{\WW_{s,t}},$$
where $\phi$ is a test function. Both integrals are well-defined because $\phi$ is smooth. The homogeneities of $\aWd$ and $\aWWd$ are negative, so they are mapped with distributions. On the other hand, $\an$ and $\aW$ have  positive homogeneities, so we can approximate them in $s$ with functions
as
$$\Pi_s(\an)(t) \defeq 1,\qquad \Pi_s(\aW)(t) \defeq W_{s,t}.$$

Now, we define $\Gamma_{s,t}(\atau)$ for every $\beta\in A$ and $s,t\in [0,T]$ and $\atau\in \modspace_\beta$ . According to Definition  \ref{model} : $\Pi_s\Gamma_{s,t}(\atau)(\phi) = \Pi_t(\atau)(\phi)$. Moreover, following  Definition \ref{regstr}, $\Gamma_{s,t}$ should be a linear combination of elements of homogeneity lower than $\atau$ and with the coefficient $1$ in front of $\atau$. First, it seems obvious to set $\Gamma_{s,t}(\an)= \an$, because $\an$ represents a constant. Then we look for $\Gamma_{s,t}(\aW)=\aW + a_{s,t}\an$ as a function where $a_{s,t}$ has to be determined. If it is not enough, we would look for $\Gamma_{s,t}(\aW)$ with more elements of our structure~$T$.
By linearity $$\Pi_s(\aW+a_{s,t}\an)(u)= W_{s,u}+a_{s,t},$$ so we want that $W_{s,u}+a_{s,t} = \Pi_t(\aW)(u)=W_{t,u}.$ Finally, we have to choose $a_{s,t}=W_{t,s}$. Given that $\aWd$ has the lowest homogeneity of our structure, we set $\Gamma_{s,t}(\aWd)=\aWd$ in order to respect the last item of  Definition \ref{regstr}. With the same reason as for $\aW$ and using the Chen's relation of Defintion~\ref{def_rough_path}, we find that $\Gamma_{s,t}(\aWWd)=\aWWd + W_{t,s}\aWd$ (see the proof of Definition~\ref{roughStruDef}).

All we did here is in one dimension.
With the same arguments we can find the regularity structure of a rough path in $\RR^n$. 

\begin{Def}
\label{roughStruDef}
For $\alpha\in (1/3,1/2]$, given a rough path $\rW=(W,\WW)\in \rC^\alpha$ which take value in $\RR^n\bigoplus\left( \RR^n\otimes\RR^n\right)$. We define the \emph{regularity structure of rough paths} $\TT^r=(\indexSet^r,\modspace^r,\grou^r)$\nomenclature{$\TT^p$}{Regularity structure of the rough paths} and the model associated $M^r=(\Pi^r,\Gamma^r)$\nomenclature{$\TT^p$}{Model associated to the regularity structure of the rough paths} as
\begin{enumerate}[i)]
\item \label{enu:index}Index set $\indexSet^r\defeq\{\alpha-1,2\alpha-1,0,\alpha\}.$
\item \label{enu:mod_space} Model space $\modspace^r\defeq\modspace^r_{\alpha-1}\bigoplus\modspace^r_{2\alpha-1}\bigoplus
\modspace^r_0\bigoplus\modspace^r_\alpha,$
with $$\modspace^r_{\alpha-1}\defeq\vect{\aWd^i,i=1,\cdots,n},\quad \modspace^r_{2\alpha-1}\defeq\vect{\aWWd^{i,j},i,j=1,\cdots,n},$$
$$\modspace^r_{0}\defeq\vect{\an},\quad \modspace^r_{\alpha}\defeq\vect{\aW^i,i=1,\cdots,n}.$$
\item \label{enu:Gamma}For $i,j$  integers between $1$ and $n$, $h\in\RR^n$ and  $\Gamma^r_h$ in the structure group $\grou^r$, the following relations hold 
$$\Gamma^r_h(\ab{\dot{W}^i}) \defeq \ab{\dot{W}^i},\qquad \Gamma^r_h(\ab{\dot{\WW}^{i,j}})\defeq\ab{\dot{\WW}^{i,j}}+h^i\ab{\dot{W}^j} ,$$
$$\Gamma^r_h(\an)\defeq\an,~\mathrm{and}\quad \Gamma^r_h(\ab{W^i})\defeq\ab{W^i}+h^i\an.$$
\item \label{enu:Pi} For $i,j$ two integers between $1$ and $n$, for $s,t\in [0,T],$
$$\Pi^r_s(\ab{\dot{W^i}})(\phi) \defeq \int \phi(t)\di{W^i}_t,\qquad \Pi^r_s(\ab{\dot{\WW}^{i,j}})(\phi)\defeq\int \phi(t)\di{\WW^{i,j}_{s,t}},$$
$$\Pi^r_s(\ab{1})(t)\defeq 1,\qquad \Pi^r_s(\ab{W^i})(t)\defeq W_{s,t}^i,$$
where $\phi$ is a test function.
\item \label{enu:gamma_s_t}For $s,t\in \RR$, $\Gamma^r_{s,t}\defeq \Gamma^r_{|h=W_{t,s}}.$
\end{enumerate}
\end{Def}

\begin{proof}
Checking that this definition respects the definitions of a regularity structure (Definition \ref{regstr}) and of a model (Definition \ref{model}) is straightforward. 

Here we only show where Chen's relation of  Definition \ref{def_rough_path} is fundamental to show that the algebraic condition of  Definition \ref{model} : $\Pi^r_s\Gamma^r_{s,t}\aWWd=\Pi^r_t\aWWd$ holds.

According to the definition above $\Gamma^r_{s,t}\aWWd^{i,j} = \ab{\dot{\WW}^{i,j}}+h^i\ab{\dot{W}^j}$. So we have
\[\Pi^r_s(\Gamma^r_{s,t}\aWWd^{i,j})(\phi)=\int \phi(u)\di{\WW^{i,j}_{s,u}}+W_{t,s}^i\int \phi(u)\di{W_{s,u}^j}.\]
In differentiating Chen's relation  $\WW^{i,j}_{s,u}=\WW^{i,j}_{s,t}+\WW^{i,j}_{t,u}+W_{s,t}^iW_{t,u}^j$ with respect to $u$ we get
$\di{\WW^{i,j}_{s,u}}=\di{\WW^{i,j}_{t,u}}+W_{s,t}^i\di{W_{t,u}^j}$.
It follows that
\[\Pi^r_s(\Gamma^r_{s,t}\aWWd^{i,j})(\phi)=\int \phi(u)\di{\WW^{i,j}_{t,u}}+W_{s,t}^i \int \phi(u)\di{W^j_{s,u}}+W_{t,s}^i\int \phi(u)\di{W_{s,u}^j}.\]
Finally $\Pi^r_s(\Gamma^r_{s,t}\aWWd^{i,j})(\phi)=\int \phi(u)\di{\WW^{i,j}_{t,u}}=\Pi^r_t\aWWd$, which is the algebraic condition required. 
\end{proof}

%%%%%%%%%% THE SPACE OF MODELLED DISTRIBUTION AND THE RECONSTRUCTION OPERATOR %%%%%%%%%%%%%%%%%%%%%%%%%%%%
\section{Modelled distributions}
\label{sec_modelledDistribution}
\subsection{Definition and the reconstruction operator}
We have defined a regularity structure. We now introduce the space of functions from $[0,T]$ to $\modspace$, the \emph{model space} of a regularity structure. These abstract functions should represent at each point of $[0,T]$, a \textquotedblleft Taylor expansion\textquotedblright\ of a real function.

We showed in Section \ref{roughstru} how to build an abstract function $\aY(t) = y_t\an + y_t'\aW$  which represents the expansion of a real controlled rough path $(y,y')$ at a point $t$. The most important result of the theory of regularity structures is to show how to build a real function or distribution from an abstract function. Namely, given an approximation of a function at each time, how to reconstruct \textquotedblleft continuously\textquotedblright\ the function. This is given by the \emph{reconstruction map} theorem.

\begin{Def}
\label{defMod}
Given a regularity structure $(\indexSet,\modspace,\grou)$  and a model $M=(\Pi,\Gamma)$, for $\gamma\in\RR$ we define the space $\model$\nomenclature{$\model$}{Modelled distributions of regularity $\gamma$} of \emph{modelled distributions} as functions $\af : [0,T]\rightarrow \modspace_{<\gamma}\defeq\bigoplus_{\beta<\gamma}\modspace_\beta$
\nomenclature{$\modspace_{<\gamma}$}{Elements of $\modspace$ of homogeneity strictly less than $\gamma$} such that for all $s,t\in [0,T]$ and for all $\beta<\gamma,$
$$\norme{\af(s)-\Gamma_{s,t}(\af(t))}_{\beta}\leq C(T)\abs{t-s}^{\gamma-\beta},$$
where $C(T)$ is a constant which depends only on $T$.

Recalling that $\norme{\cdot}_\beta$ is the norm of the component in $\modspace_\beta$, we define by $$\norme{\af}_{\gamma,T}\defeq\sup_{s\ne t\in [0,T]}\sup_{\beta<\gamma} \frac{\norme{\af(t)-\ga(\af(s))}_{\beta}}{\abs{t-s}^{\gamma-\beta}} $$ a semi-norm on the space $D^\gamma_M$. It is also possible to consider the norm
 $$\norme{ \af}_{\gamma,T}^*\defeq\sup_{t\in  [0,T]}\sup_{\beta<\gamma}\norme{\af(t)}_{\beta} + \norme{\af}_{\gamma,T}.$$\nomenclature{$\norme{ \af}_{\gamma,T}^*$}{Norm of the modelled distribution space}
  Moreover $\norme{\cdot}^*_{\gamma,T}$ is equivalent to  $$\sup_{\beta<\gamma}\norme{\af(0)}_{\beta} + \norme{\af}_{\gamma,T},$$ so from now we  use these two norms without distinction.

\end{Def}

\begin{Rem}
For a fixed model $M$, the modelled distributions space $D^{\gamma}_M$ is a Banach space with the norm~$\norme{~}_{\gamma,T}^*$. 
\end{Rem}

\begin{Rem}
We choose the same notation  for the semi-norm on $D^{\gamma}_M$ as on $\DD^\gamma_W$ (the space of modelled distributions and on $C^\alpha$ (the space of Hölder functions or distributions).

So when $\af\in D^\gamma_M$, we have to understand $\norme{\af}_{\gamma,T}$ with Definition \ref{defMod} but when $f\in C^\alpha$, $\norme{f}_{\alpha,T}$ is the Hölder norm of Definition \ref{holder1} (for functions $\alpha>0$) or \ref{holder2} (for distributions $\alpha<0$).
\end{Rem}

\begin{Rem}
The modelled distribution space $D^\gamma_M$ can be thought of as abstract $\gamma$-Hölder functions. Indeed, for an integer $p$  and $\delta\in [0,1)$ such that $\gamma=p+\delta$, if $f$ is a smooth function
$$\left|f(x) - \sum_{k=0}^p \frac{f^{(k)}(y)}{k!}(y-x)^k\right|\leq C \abs{t-s}^\delta,$$
according to the Taylor's inequality. Hence, Definition \ref{defMod} of modelled distributions has to be seen as an extension of the Taylor inequality in a no classical way.
\end{Rem}

Now we are able to outline the main theorem of the theory of  regularity structures which given a modelled distribution allows us to build a \textquotedblleft real\textquotedblright\ distribution approximated at each point by the modelled distribution.

\begin{The}[Reconstruction map]
\label{the_reconstruction_map}
Given a regularity structure $\TT=(\indexSet,\modspace,\grou)$ and a model $M=(\Pi,\Gamma)$, 
for a real $\gamma$\nomenclature{$\Rc$}{The reconstruction operator} $>$ $\alpha_* = \min{\indexSet}$ and an integer $r>|\alpha_*|$ there is a linear continuous map $\Rc:D^\gamma_M\rightarrow C^{\alpha_*}$ such that for all $\af\in D^{\gamma}_M$,
\[\abs{\left[\Rc(\af)-\Pi_s(\af(s))\right](\phi_s^\lambda)}\leq C\norme{\Pi}_{\gamma,T}\norme{\af}_{\gamma,T}^* \lambda^\gamma,\label{eq:estimation_reconstruction_theorem}\]
where  $C$ depends uniformly over $\phi\in B_r$, $\lambda\in (0,1]$, $s\in [0,T]$.

Moreover if $\gamma>0$, the bound \eqref{eq:estimation_reconstruction_theorem} defined $\Rc(\af)$ uniquely.

If $(\tilde{\Pi},\title{\Gamma})$ is an other model for $\TT$ and $\tilde{\Rc}$ the reconstruction map associated to the model, we have the bound
\begin{align}
\nonumber
&|\Rc(\af)-\tilde{\Rc}(\tilde{\af})-\Pi_s(\af(s))+\tilde{\Pi}_s(\tilde{\af}(s))](\eta_s^\lambda)|\\&\leq C \left(\norme{\tilde{\Pi}}_{\gamma,T}\norme{f-\tilde{\af}}_{\gamma,T}^*+\norme{\Pi-\tilde{\Pi}}_{\gamma,T}
\norme{\af}_{\gamma,T}^*\right)\lambda^\gamma\label{eq:diff_R_Rtilde},
\end{align}
where  $C$ depends uniformly over $\phi\in B_r$, $\lambda\in (0,1]$, $s\in [0,T]$, as above.
\end{The}

\begin{proof}
The proof uses the wavelet analysis in decomposing the function $\af$ in a smooth wavelet basis. The proof requires many computation. A complete one can be found in \cite{hairer} and a less exhaustive one is in \cite{frizhairer}.
The construction of $\Rc(\af)$ is the following. We define a sequence $(\Rc^j(\af))_{j\in\NN}$ such that
\[\Rc^j(\af)\defeq \sum_{k\in\ZZ}\Pi_{k/2^j}(\af(k/2^j))(\phi_k^j)\phi_k^j,\]
where $\phi_k^j$ is defined in Definition \ref{def:wavelet} with a regularity at almost $r$. Then, we show that  $\Rc^j(\af)$ converges weakly to a distribution $\Rc(\af)$ which means that $\Rc^j(\af)(\eta)$ converges to $\Rc(\af)(\eta)$ for all $\eta\in C^r_0$. And we show that the bound~\eqref{eq:estimation_reconstruction_theorem} holds.
\end{proof}

\begin{Rem}
It can be proved that if for all $s\in [0,T]$ and $\atau\in \modspace$, $\Pi_s\atau$ is a continuous function then $\Rc(\af)$ is also a continuous function such that
\[\Rc(\af)(s)=\Pi_s(\af(s))(s).\]
\end{Rem}

\begin{Cor}
With the same notation as in  Theorem \ref{the_reconstruction_map}, for every $\gamma>0$, there is a constant $C$  such as
$$\norme{\Rc(\af)}_{\alpha,T}\leq C \norme{\Pi}_{\gamma,T}\norme{\af}_{\gamma,T}^*.$$
\end{Cor}

\begin{proof}
According to  Theorem \ref{the_reconstruction_map}, for $\phi\in B_r$,
$$\frac{\abs{\Rc(\af)(\phi_s^\lambda)}}{\lambda^{\alpha}}\leq\frac{\abs{\Pi_s(\af(s))(\phi_s^\lambda)}}{\lambda^\alpha}+ C\norme{\Pi}_{\gamma,T}\norme{\af}_{\gamma,T}^* \lambda^{\gamma-\alpha},$$
and according to the Definition \ref{model}, 
$$\frac{\abs{\Pi_s(\af(s))(\phi_x^\lambda)}}{\lambda^\alpha}\leq \norme{\Pi}_{\gamma,T}\norme{\af}_{\gamma,T}.$$
So finally 
\begin{eqnarray*}
\norme{\Rc(\af)}_{\alpha,T}&\leq &\norme{\Pi}_{\gamma,T}\norme{\af}_{\gamma,T} + C\norme{\Pi}_{\gamma,T}\norme{\af}_{\gamma,T}^* \lambda^{\gamma-\alpha}\\
&\leq & C\norme{\Pi}_{\gamma,T}\norme{\af}_{\gamma,T}^*,
\end{eqnarray*}
which, by letting $\lambda$ going to $0$ proves the inequality.
\end{proof}

\subsection{Modelled distribution of controlled rough paths}
We reformulate the definition of a controlled rough path in the regularity structures framework.
\begin{Def}
\label{def_abs_mod}
Given $(W,\WW)\in \rC^\alpha$, $(y,y')\in \DD^{2\alpha}_W$, the rough path regularity structure $(\indexSet^r,\modspace^r,\grou^r)$ and $M^r=(\Pi^r,\Gamma^r)$ the model associated (cf. Definition \ref{roughStruDef}), we define a modelled distribution $\aY\in D^{2\alpha}_{M^r}$ such that $$\aY(t) = y_t\an + y'_t\aW,\quad \forall t\in[0,T].$$
The space $D^{2\alpha}_{M^r}$\nomenclature{$D^{2\alpha}_{M^r}$}{Modelled distributions of the controlled rough paths} is the space of the \emph{ modelled distributions of the controlled rough paths}.
\end{Def}

\begin{Rem}
This definition is a particular case of modelled distributions of Definition~\ref{defMod}.
\end{Rem}

\begin{proof}
Let check that $\aY$ is in $D^{2\alpha}_{M^r}$. For every $s,t\in [0,T]$, 
\begin{eqnarray*}
\aY(t)-\Gamma^r_{t,s}(\aY(s))& = & \aY(t)-\Gamma^r_{t,s}(y_s\an+y'_s\aW)\\
&=& \aY(t)-(y_s\an+ y'_s\aW + y'_sW_{s,t}\an),
\end{eqnarray*}
using the Definition \ref{roughStruDef}.
Then, we have \begin{eqnarray*}
\norme{\aY(t)-\Gamma^r_{t,s}(\aY(s))}_0 =\norme{y(t)-y(s)-y'(s)W_{s,t}}\leq  C \abs{t-s}^{2\alpha},
\end{eqnarray*}
according to the definition \ref{contrough} of controlled rough paths.
Besides, 
$$
\norme{\aY(t)-\Gamma^r_{t,s}(\aY(s))}_\alpha  = \norme{y'(t)-y'(s)} \leq  C \abs{t-s}^{\alpha},$$
which proves that $\aY\in D^{2\alpha}_{M^r}$.
\end{proof}

\begin{Pro}
\label{pro_equivalence_control_modelledcontrol_roughpath}
With the notations of Definition \ref{def_abs_mod}, the application  $(y,y')\in\DD^{2\alpha}_W\mapsto \aY\in D^{2\alpha}_{M^r}$
is an isomorphism and the norms $\norme{y}_{\infty,T}+\norme{y'}_{\infty,T}+\norme{(y,y')}_{2\alpha,T}^W$ and $\norme{\aY}_{2\alpha,T}^*$ are equivalent.
\end{Pro}

\begin{proof}
We prove the only equivalence between the two norms.

With the notation of Definition \ref{contrough}, we recall that 
\[y_{s,t}=y_s'W_{s,t}+R^y_{s,t},\]
and that $\norme{(y,y')}_{2\alpha,T}^W=\norme{y'}_{\alpha,T}+\norme{R^y}_{2\alpha,T}$.
Then according to the previous proof and Definition \ref{defMod},
\[\norme{\aY}_{2\alpha,T}=\sup \left\{\norme{y'}_{\alpha,T},\norme{R^y}_{2\alpha,T}\right\}.\]
So we have $\norme{\aY}_{2\alpha,T}\leq \norme{(y,y')}_{2\alpha,T}^W $ and
$\norme{(y,y')}_{2\alpha,T}^W\leq 2 \norme{\aY}_{2\alpha,T}$.
In adding the terms   $\norme{y}_{\infty,T}+\norme{y'}_{\infty,T}$ to each semi-norms, we obtain the result.
\end{proof}

%%%% ROUGH PATH INTEGRAL %%%%%%%%%%%%%%%%%%%%%%%%%%%%%%%%%%%%%%%%%%%%%%%%%%
\section{Rough path integral with the reconstruction map}
\label{sec_roughInt}
The power of the theory of regularity structures is to give a sense in some cases of a product of distributions. Indeed, it is not possible in general to extend the natural product between functions to the space of distributions.

To build the controlled rough path integral of  Theorem \ref{classic_rough_int}, with the theory of  regularity structures we need to give a meaning to the product between $y$ and $\dot{W}$, where $\dot{W}$ is a distribution.
We start by giving a meaning to the abstract product between $\aY$ and $\aWd$.
When the product has good properties, we use the reconstruction map (Theorem \ref{the_reconstruction_map}) to define a  \textquotedblleft real\textquotedblright\  multiplication.

\begin{Def}[Multiplication in the model space]
\label{def_multiplication}
Given a regularity structure $(\indexSet,\modspace,\grou)$,
 we say that the  continuous bilinear map  \nomenclature{$\mul$}{Abstract product}$\mul:\modspace^2\rightarrow \modspace$ defines a \emph{multiplication}  (\emph{product}) on the model space $\modspace$ if
\begin{itemize}
\item For all $\atau\in \modspace$, on has $\an\star\atau=\an,$
\item For every $\atau\in T_\alpha$ and $\ab{\sigma}\in T_\beta,$ on has
 $\atau\mul\ab{\sigma}\in T_{\alpha+\beta}$, if $\alpha+\beta\in \indexSet$ and
 $\atau\mul\ab{\sigma}=0$ if $\alpha+\beta\notin\indexSet$.
 \item For every $\atau\in T_\alpha$, $\ab{\sigma}\in T_\beta$ and $\Gamma\in\grou$, $\Gamma(\atau\mul\ab{\sigma})=\Gamma(\atau)\mul\Gamma(\ab{\sigma})$.
\end{itemize}
We denote by $|\atau|$\nomenclature{$|\atau|$}{The homogeneity of the symbol $\atau$} the homogeneity $\alpha$ of the symbol $\atau$. The last item of the definition can be rephrased as $|\atau\star\ab{\sigma}|=|\atau|+|\ab{\sigma}|$.
\end{Def}

\begin{Rem}
\label{rem_product_rough_path}
For example in the following Theorem \ref{the_sewing_lemma}, we define within the regularity structure of rough paths the multiplication described in the table below:

\begin{tabular}{|c|c|c|c|c|}
  \hline
   $\star$ & $\aWd$ & $\aWWd$ & $\an$ & $\aW$ \\
  \hline
  $\aWd$ & $~$ & $~$ & $\aW$ & $\aWWd$ \\
  \hline
  $\aWWd$ & $~$ & $~$ &$\aWWd$ & $~$\\
  \hline  
  $\an$ &$\aWd$ &$\aWWd$ &$\an$ &$\aW$\\
  \hline  
  $\aW$ &$\aWWd$ & $~$ &$\aW$  & $~$\\
  \hline
\end{tabular}
\end{Rem}

We are now able to build the rough integral with the reconstruction theorem
(Theorem~\ref{the_reconstruction_map}). The operator $\aI$ corresponding to
the integral of a controlled rough path against a rough path.
\begin{The}
\label{the_sewing_lemma}
We set $\alpha\in\rinter$. There is a linear map $\aI:\rmodel \rightarrow \ahol$  such that  for all $\aY\in\rmodel,$ $\aI(\aY)(0)=0$ and such that the map $L$ defined by 
$$\forall t\in[0,T],~L(\aY)(t)\defeq\aI(\aY)(t)\an + \langle \aY(t),\an \rangle\aW$$  is linear and continuous from $\rmodel$ into itself. The symbol $\langle \cdot,\an \rangle$ \nomenclature{$\langle \cdot,\an \rangle$}{The projection onto $\an$}denotes the coordinate along $\an$.
\end{The}

\begin{Rem}
Recalling that if $\aY\in \rmodel$, according to the Definition \ref{def_abs_mod} there is $(y,y')\in\DD^{2\alpha}_W$ such that
\[\aY(t)=y_t\an+y'_t\aW,\]
we show in the proof of the Theorem \ref{the_sewing_lemma} that
\[I(\aY)(t)=\int_0^t y_s\di{\rW_s},\]
where $\int_0^t y_s\di{\rW_s}$ is defined in Theorem~\ref{classic_rough_int}. Thus $L$ is the equivalent in the modeled distribution space of the map
\[(y,y')\in\DD^{2\alpha}_W\mapsto \left(\int_0^\cdot y_s\di{\rW_s},y_\cdot\right)\in \DD^{2\alpha}_W. \]
\end{Rem}

\begin{Rem}
The proof of the existence of $I$ is the same as  in  Theorem~\ref{classic_rough_int} (classical sewing lemma). But we show how Theorem~\ref{the_reconstruction_map} (reconstruction map) can be adapted to recover the result.
\end{Rem}

\begin{proof}
For $\aY$ in $\rmodel$, we define the point-wise product between $\aY$ and $\aWd$ as in Remark~\ref{rem_product_rough_path}, i.e $\aY(t)\mul\aWd \defeq y_t\aWd+y_t'\aWaWd$, where $\aWaWd\defeq\aW\star\aWd\defeq \aWWd$. We denote this product $\aY\aWd(t)$, to simplify the notation. Using the fact that $|\aW|+|\aWd|=2\alpha-1=|\aWWd|$ it is straightforward to check that the product is consistent with the Definition~\ref{def_multiplication}.

We check now that $\aY\aWd$ is in $D^{3\alpha-1}_{M^r}$. According to Definition~\ref{roughStruDef} item~\ref{enu:gamma_s_t}, we compute $$\ga^r\left(\aY\aWd(s)\right)=(y_s+y'_sW_{s,t})\aWd+y'_s\aWWd,$$ since $\aY\in D^{2\alpha}_{M^r}$ with Definition~\ref{def_abs_mod},
\[\norme{\aY\aWd(t)-\ga^r\left(\aY\aWd(s)\right)}_{\alpha-1}=\norme{y_{s,t}-y'_sW_{s,t}}\est \abs{t-s}^{2\alpha},
\label{eq:preuve_int_1}\]
\[\norme{\aY\aWd(t)-\ga^r\left(\aY\aWd(s)\right)}_{2\alpha-1}=\norme{y'_{s,t}}\est \abs{t-s}^{\alpha}.\label{eq:preuve_int_2}\]
Thus, by Definition \ref{defMod}, we get that $\aY\aWd\in D^{3\alpha-1}_{M^r}.$

Thus, given that $3\alpha-1>0$, we can apply the reconstruction theorem in the positive case.

 So there is a unique distribution $\Rc{(\aY\aWd)}$ in $C^{\alpha-1}$ such that for every $s\in [0,T],$ $\lambda>0$ and every localized test function $\eta_s^\lambda$ of Definition~\ref{def:test_functions}, 
\[\left|\Rc(\aY\aWd)(\eta_s^\lambda)-y_s\int \eta_s^\lambda (u)\di{W_u}-y_s'\int \eta_s^\lambda(u)\di{\WW_{s,u}}\right|\leq C\norme{\eta}_{C^1} \lambda^{3\alpha-1},\label{eq:reconstruction_integrale_rough}\]
where we use relations of the item \ref{enu:Pi} of  Definition \ref{roughStruDef}.

We define with Lemma~\ref{lem_lien_distribution_fonction} the operator $I:\rmodel\rightarrow C^\alpha$ such that $I(\aY)\in C^\alpha$ is associated to $\Rc(\aY\aWd)$. It means that $I\left(\aY\right)(0)\defeq 0$ and $\langle I\left(\aY\right),\eta'\rangle\defeq -\langle\Rc(\aY\aWd),\eta\rangle$.  More precisely, we have for $|t-s|\leq 1$,
\[{I(\aY)}_{s,t}=\sum_{k\in I_l}
  \langle\Rc(\aY\aWd),\phi^{l}_k\rangle\int_s^t\phi^{l}_k+\sum_{j\geq
    l}\sum_{k\in I_l}
  \langle\Rc(\aY\aWd),\psi^j_k\rangle\int_s^t\psi^j_k.\]

  Moreover, according to Theorem~\ref{def:wavelet}, we can  choose
  the integer $l$ such that $2^{-l}\leq |t-s|<2^{-l+1}$.

%%%%%%%%%%%%%%%%%%%%%%%%%%%%%%%%
We have 
\begin{align}
\label{eq:wavelet_sum_estimation_1}
I(\aY)_{s,t}-y_sW_{s,t}-y'_s\WW_{s,t}&=\sum_{k\in I_l} \langle\Rc(\aY\aWd)-\Pi_{s}(\aY\aWd(s)),\phi^{l}_k\rangle\int_s^t\phi^{l}_k\\
&+\sum_{j\geq l}\sum_{k\in I_j} \langle\Rc(\aY\aWd)-\Pi_{s}(\aY\aWd(s)),\psi^j_k\rangle\int_s^t\psi^j_k.
\label{eq:wavelet_sum_estimation_3}
\end{align}

We have
\begin{align}\nonumber
  \langle\Rc(\aY\aWd)-\Pi_{s}(\aY\aWd(s)),\psi^j_k\rangle
  &=\langle\Rc(\aY\aWd)-\Pi_{k/2^j}(\aY\aWd(k/2^j)),\psi^j_k\rangle\\
&+\langle\Pi_{k/2^j}(\aY\aWd(k/2^j))-\Pi_{s}(\aY\aWd(s)),\psi^j_k\rangle\label{eq:rc-pi}.
\end{align}
The first term of the right side of \eqref{eq:rc-pi} is bounded by \eqref{eq:estimation_reconstruction_theorem},
\begin{align}
|\langle\Rc(\aY\aWd)-\Pi_{k/2^j}(\aY\aWd(k/2^j)),\psi^j_k\rangle|\leq
  C2^{-j/2}2^{j(1-3\alpha)}\label{eq:diff_pi_0}.
  \end{align}
For bounding the second term of the right side  of \eqref{eq:rc-pi} we use the algebraic relations between $\Pi$ and $\Gamma$ as well as the relations \eqref{def_model_ineq},
\begin{align*}
&\langle\Pi_{k/2^j}(\aY\aWd(k/2^j))-\Pi_{s}(\aY\aWd(s)),\psi^j_k\rangle\\
&=\langle\Pi_{k/2^j}\left(\aY\aWd(k/2^j)-\Gamma_{k/2^j,s}\aY\aWd(s)\right),\psi_k^j\rangle.
\end{align*}
Yet $\aY\aWd\in D^{3\alpha-1}_{M^r}$, so with \eqref{eq:preuve_int_1} and \eqref{eq:preuve_int_2}, we have $$\norme{\aY\aWd(k/2^j)-\Gamma_{k/2^j,s}\aY\aWd(s)}_\beta\leq C |k/2^j-s|^{3\alpha-1-\beta},$$ for $\beta\in\{2\alpha-1,\alpha-1\}$. Finally, we obtain with the bounds \eqref{def_model_ineq},
\begin{align}
&\left|\langle\Pi_{k/2^j}(\aY\aWd(k/2^j))-\Pi_{s}(\aY\aWd(s)),\psi^j_k\rangle\right|\\
  &\leq\sum_{\beta\in\{2\alpha-1,\alpha-1\}} 2^{-j\beta-j/2}\left|\frac{k}{2^j}-s\right|^{3\alpha-1-\beta}.
\label{eq:diff_pi_2}
\end{align}

Moreover, we have  $k/2^j\in [-c/2^j-s,c/2^j+t]$ for all terms that
are non-vanishing in \eqref{eq:wavelet_sum_estimation_1} and
\eqref{eq:wavelet_sum_estimation_3}. Since $j\geq l$ in the sums and
that we assume $2^{-j}\leq 2^{-l}\leq |t-s|< 2^{-l+1}$, we have
\begin{align}
  \label{eq:t-s}
  \left|\frac{k}{2^j}-s\right|\leq C |t-s|,
\end{align}
for all non-vanishing terms in the sums \eqref{eq:wavelet_sum_estimation_1} and
\eqref{eq:wavelet_sum_estimation_3}.

Firstly we bound \eqref{eq:wavelet_sum_estimation_1}. On the one hand,
using
\eqref{eq:diff_pi_0},
\eqref{eq:diff_pi_2}, \eqref{eq:t-s} and the fact that
$|t-s|<2^{-l+1}$, we obtain
\begin{align}
\label{eq:R_pi_psi}
  \left|\langle\Rc(\aY\aWd)-\Pi_{s}(\aY\aWd(s)),\psi^j_k\rangle\right|
  \leq C 2^{-l/2}2^{-l(3\alpha-1)}.
\end{align}
On another hand, we have
\begin{align}
  \nonumber
  \left|\int_s^t\phi_k^l\right|&\leq
                                 2^{l/2}|t-s|\sup_{t\in\RR}\norme{\phi(t)}\\
\label{eq:int_s_t_phi_l}  &\leq C 2^{-l/2}.
\end{align}
Thus, because there is  only a finite number of terms 
independent  on $l$ that contribute to the sum
\eqref{eq:wavelet_sum_estimation_1}, we obtain with
\eqref{eq:R_pi_psi} and \eqref{eq:int_s_t_phi_l}
the following bound on \eqref{eq:wavelet_sum_estimation_1}:
\begin{align}
  \label{eq:decomposition_sum_1}
\left|\sum_{k\in I_l}
  \langle\Rc(\aY\aWd)-\Pi_{s}(\aY\aWd(s)),\phi^{l}_k\rangle\int_s^t\phi^{l}_k\right|\leq
  C2^{-l3\alpha}\leq C|t-s|^{3\alpha},
\end{align}
where $C$ does not depends on $l$.

Now, we bound \eqref{eq:wavelet_sum_estimation_3}. On the one hand, using
\eqref{eq:diff_pi_0},
\eqref{eq:diff_pi_2}, \eqref{eq:t-s}, we have
for $j\geq l$,
\begin{align}
\nonumber
&|\langle\Rc(\aY\aWd)-\Pi_{s}(\aY\aWd(s)),\psi^j_k\rangle|\\
&\leq C2^{-j/2}\left[2^{j(1-3\alpha)}+|t-s|^{2\alpha}2^{-j(\alpha-1)}+|t-s|^{\alpha}2^{-j(2\alpha-1)}\right]\label{eq:diff_pi_final}.
\end{align}

On an other hand, we observe that  \[\left|\sum_{k\in I_j}\int_s^t\psi_k^j\right|\leq C2^{-j/2}\label{eq:int_psi},\] because a primitive of $\psi$ has a compact support and the fact that $\int \psi=0$. Then, combining \eqref{eq:diff_pi_final} and \eqref{eq:int_psi} we obtain,
\begin{align}
\nonumber
&\left |\sum_{j\geq l}\sum_{k\in I_j} \langle\Rc(\aY\aWd)-\Pi_{s}(\aY\aWd(s)),\psi^j_k\rangle\int_s^t\psi^j_k\right|\nonumber\\
\nonumber
&\leq \sum_{j\geq l}2^{-3j\alpha}+|t-s|^{2\alpha}2^{-j\alpha}+|t-s|^{\alpha}2^{-j2\alpha}\\\nonumber
&\leq C 2^{-3l\alpha}+|t-s|^{2\alpha}2^{-l\alpha}+|t-s|^{\alpha}2^{-l2\alpha}\\
&\leq C |t-s|^{3\alpha}.\label{eq:estimation_r-pi_2}
\end{align}
With \eqref{eq:decomposition_sum_1} and \eqref{eq:estimation_r-pi_2} we obtain the bound of the left hand side of \eqref{eq:wavelet_sum_estimation_1},
\[|I(\aY)_{s,t}-y_sW_{s,t}-y'_s\WW_{s,t}|\leq C|t-s|^{3\alpha}.\label{eq:I-yW}\]

To show that $L(\aY)$ is in $\rmodel$, we compute 
$\ga^r{\left(L(\aY)(s)\right)}=(I(\aY)(s)+y_sW_{s,t})\an+y_s\aW$ and we use the estimation \eqref{eq:I-yW}. Thus, we have
\begin{align}
\nonumber
\norme{L(\aY)(t)-\ga^r{\left(L(\aY)(s)\right)}}_{0} &=\norme{I(\aY)(t)-I(\aY)(s)-y_sW_{s,t}}\\
&\leq\norme{y'}_{\infty,T}\norme{\WW}_{2\alpha,T}\abs{t-s}^{2\alpha}+C\abs{t-s}^{3\alpha}, \label{eq:norm_0_L}
\end{align}
\[\text{and} \norme{L(\aY)(t)-\ga{\left(L(\aY)(s)\right)}}_{\alpha} =\norme{y_{s,t}}\\
\leq  \norme{y}_{\alpha}\abs{t-s}^{\alpha},\label{eq:norm_1_L} 
\]
which proves that $L(\aY)$ is in $\rmodel$.

It remains to prove the continuity of $L$. According to \eqref{eq:estimation_reconstruction_theorem}, the constant $C$ in \eqref{eq:I-yW} is proportional to $\norme{\aY}_{\gamma,T}^*$. So we have,
\begin{align*}
|I(\aY)_{s,t}-y_sW_{s,t}|\leq \norme{y'}_\infty\norme{\WW}_{2\alpha,T}|t-s|^{2\alpha}+C\norme{Y}_{2\alpha,T}^*|t-s|^{3\alpha},
\end{align*}
which allows with the previous computation \eqref{eq:norm_0_L} and \eqref{eq:norm_1_L}  to bound
\[\norme{L(\aY)}_{2\alpha,T}^*\leq C\norme{\aY}_{2\alpha,T}^*.\]
This concludes the proof.
\end{proof}

%%%%%%%%%%%%%%%%%%%%%%%%%%%%%%%%%%%%%%%%%%%%%%%%%%%%%%%%%%%%%%%%%%%%%%%%%%%%
\section{Existence of a rough path lift}
\label{sub_rough_lift}
As an application of the reconstruction operator in the case $\gamma\leq 0$, we prove  Theorem~\ref{levyArea} which states that for any $W\in C^\alpha$ ($\alpha\in (1/3,1/2]$) with values in $\RR^n$, it exists a rough path lift $\WW$ and that the map $W\mapsto \rW$ is continuous from $C^\alpha$ to $\rC^{\alpha}$.

\begin{proof}[Proof (Theorem \ref{levyArea})]
We consider the regularity structure $(\indexSet^e,\modspace^e,\grou^e)$ such that
$\indexSet^e=\{\alpha-1,0\}$, $\modspace^e =\vv\langle\aWd^{i},i=1,\dots ,n\rangle\bigoplus\vv\langle\an\rangle$ and for $\Gamma_h^e\in \grou$,
$\Gamma_h^e(\aWd)=\aWd,$ $\Gamma_h^e(\an)=\an.$ We associate the model $M^e=(\Pi^e,\Gamma^e)$ such that for every $s,t\in [0,T]$, $\eta\in B_1$ 
$$\Pi_s^e(\aWd)(\eta)\defeq\int \eta(t)\di{W_t},\qquad \Pi^e_s(\an)(t)\defeq 1,$$ and $\Gamma^e_{s,t}\defeq\Gamma^e_{W_{t,s}}$.

For $0\leq s\leq t\leq 1$, and integers $0\leq i,j\leq n$, the modelled distribution $\aWWd$ given by $\ab{\dot{\WW}^{i,j}}(s) \defeq W_{s}^i\aWd^{j}$ is in $D^{2\alpha-1}_M$. Indeed $\ab{\dot{\WW}^{i,j}}(t)-\Gamma_{t,s}^e\left(\ab{\dot{\WW}^{i,j}}(s)\right) = W_{t}^i\aWd^j - W_{s}^i\aWd^j=W_{s,t}^i\aWd^j,$ then $$\norme{\ab{\dot{\WW}^{i,j}}(t)-\Gamma^e_{t,s}\left(\ab{\dot{\WW}^{i,j}}(s) \right)}_{\alpha-1}\leq \abs{t-s}^{\alpha}.$$
So, $\gamma-(\alpha-1)=\alpha$, we have $\gamma = 2\alpha -1$. We conclude using the Definition~\ref{defMod}.

Given that $\alpha\in (1/3,1/2]$, we have $2\alpha-1\leq 0$. Thus, the uniqueness of the reconstruction map does not hold. But, according to  Theorem \ref{the_reconstruction_map}, there exists $\Rc(\aWWd)\in C^{\alpha-1}$ such that
\[|[\Rc(\aWWd)-\Pi_s^e(\aWWd)](\eta_s^\lambda)|\leq C\lambda^{2\alpha-1},\label{eq:WW_dot}\]
where $\eta\in B_1$. With  Lemma~\ref{lem_lien_distribution_fonction}, we define $z\in C^\alpha$ as  the primitive of $\Rc(\aWWd)$ such that $z(0)=0$. Moreover, we have for all $s,t\in [0,1]$,
\[z_{s,t}=\sum_{k\in I_l} \langle\Rc(\aWWd),\phi^{l}_k\rangle \int_s^t\phi^{l}_k+\sum_{j\geq l}\sum_{k\in I_j} \langle \Rc(\aWWd),\psi^j_k\rangle \int_s^t\psi^j_k,\label{eq:almost_WW}\]
and
\[W_{s,t}=\sum_{k\in I_l} \langle\Pi^e_s(\aWd),\phi^{l}_k\rangle \int_s^t\phi^{l}_k+\sum_{j\geq l}\sum_{k\in I_j} \langle \Pi^e_s(\aWd),\psi^j_k\rangle \int_s^t\psi^j_k,\]
which yields to
\[W_s\otimes W_{s,t}=\sum_{k\in I_l} \langle\Pi^e_s(\aWWd(s)),\phi^{l}_k\rangle \int_s^t\phi^{l}_k+\sum_{j\geq l}\sum_{k\in I_j} \langle \Pi^e_s(\aWWd(s)),\psi^j_k\rangle \int_s^t\psi^j_k.\]
If there is a constant $C\geq 0$ such that,
\[|z_{s,t}-W_s\otimes W_{s,t}|\leq C|t-s|^{2\alpha},\label{eq:z-W}\]
 then setting $\WW_{s,t}\defeq z_{s,t}-W_{s}\otimes W_{s,t}$, the pair $(W,\WW)$ belongs to $\rC^\alpha$ according to the Definition \ref{def_rough_path}. Let us prove \eqref{eq:z-W}. We have
 \begin{align}
 \nonumber
 z_{s,t}-W_s\otimes W_{s,t}&=\sum_{k\in I_l} \langle\Rc(\aWWd)-\Pi^e_{s}(\aWWd(s)),\phi^{l}_k\rangle\int_s^t\phi^{l}_k\\&
 +\sum_{j\geq l}\sum_{k\in I_j} \langle\Rc(\aWWd)-\Pi^e_{s}(\aWWd(s)),\psi^j_k\rangle\int_s^t\psi^j_k.\label{eq:wave_z-w}
 \end{align}
 From \eqref{eq:estimation_reconstruction_theorem}, we have the bounds
 \[|\langle\Rc(\aWWd)-\Pi^e_{s}(\aWWd(s)),\phi^j_k\rangle|\leq C 2^{-j/2-j(2\alpha-1)},\label{eq:r-w-phi}\]
 and\[|\langle\Rc(\aWWd)-\Pi^e_{s}(\aWWd(s)),\psi^j_k\rangle|\leq C 2^{-j/2-j(2\alpha-1)}.\label{eq:r-w-psi}\]

Then, combining \eqref{eq:wave_z-w},\eqref{eq:r-w-phi} and \eqref{eq:r-w-psi}, we proceed as in the proof of\\ Lemma~\ref{lem_lien_distribution_fonction} to show \eqref{eq:z-W}.

It remains to show the continuity.
If there is another path $\tilde{W}\in C^\alpha$, we define as for $W$, a model $(\tilde{\Pi},\tilde{\Gamma})$, a modelled distribution $\tilde{\aWWd}$, a reconstruction map $\tilde{\Rc}$ and then $\tilde{\WW}$. By denoting
\[\Delta\Pi_{s,k/2^j}\defeq [\Pi(\aWWd(k/2^j))-\Pi(\aWWd(s))-\tilde{\Pi}(\tilde{\aWWd}(k/2^j))+\tilde{\Pi}(\tilde{\aWWd}(s))](\psi_k^j),\]
we have
\begin{align}
\nonumber
&|\Delta\Pi_{s,k/2^j}|\\
&\leq \norme{W-\tilde{W}}_{\alpha,T}\left(\norme{W}_{\alpha,T}+\norme{\tilde{W}}_{\alpha,T}\right)|s-k/2^j|2^{j/2(1-\alpha)}2^{-j/2}\label{eq:proof_existence_continuity}.	
\end{align}
According to the bounds \eqref{eq:diff_R_Rtilde}, \eqref{eq:proof_existence_continuity} and in writing
\begin{align*}
&[\Rc(\aWWd)-\Pi_s(\aWWd(s))-\tilde{\Rc}(\tilde{\aWWd})+\Pi_s(\tilde{\aWWd}(s))](\psi_k^j)\\
&=\Rc(\aWWd)-\Pi_{k/2^{j}}(\aWWd(k/2^{j}))-\tilde{\Rc}(\tilde{\aWWd})+\tilde{\Pi}_{k/2^j}(\tilde{\aWWd}(k/2^j))+\Delta\Pi_{s,k/2^j},
\end{align*}
we get
\begin{align*}
|\WW_{s,t}-\tilde{\WW}_{s,t}|&\leq C \Big[\norme{\tilde{\Pi}}_{2\alpha-1,T}
\norme{\aWWd-\tilde{\aWWd}}_{2\alpha-1,T}^*+\norme{\Pi-\tilde{\Pi}}_
{2\alpha-1,T}\norme{\aWWd}_{2\alpha-1,T}^*\\
&+\norme{W-\tilde{W}}_{\alpha,T}\Big(\norme{W}_{\alpha,T}
+\norme{\tilde{W}}_{\alpha,T}\Big)\Big]|t-s|^{2\alpha}.
\end{align*}
Yet we have, $\norme{\aWWd-\tilde{\aWWd}}_{2\alpha-1,T}^*=\norme{W-\tilde{W}}_{\infty,T}
+\norme{W-\tilde{W}}_{\alpha,T},$
and 
\[\norme{\Pi-\tilde{\Pi}}_
{2\alpha-1,T}\leq C \norme{W-\tilde{W}}_{\alpha,T}.\]
So finally,
\[\norme{\rW-\tilde{\rW}}_{\alpha,T}\leq C\norme{W-\tilde{W}}_{\alpha,T},\]
which proves the continuity.
\end{proof}

\begin{Rem}
Given that $2\alpha-1$ is negative, the uniqueness of $\WW$ does not hold, which is in accordance with Remark
\ref{uniq}.
\end{Rem}

%%%%%%%%%%%%%%%%%%%%%%%%%%%%%%%%%%%%%%%%%%%%%%%%%%%%%%%%%%%%%%%%%%%%%%%%%%%%
\section{Composition with a smooth function}
\label{sec_comp}
Before solving the general rough differential equation \eqref{rde} with the theory of regularity structures, we should give a sense of the composition of a modelled distribution with a function. Then we will be able to consider \eqref{rde} in the space of the modelled distributions.

The composition of a modelled distribution $\af\in D^\gamma_M$ with a smooth function $F$ is developed in \cite{hairer}. The author gives a general theorem which allows  the composition with an arbitrary smooth function $F$ when $\af$ takes its values in a model space $\modspace$ such that the smallest index of homogeneity is equal to $0$, \emph{i.e.} $\forall t\in\RR,~\af(t)\in \vect{\an,...} $. Thus, it is possible to define the composition as a Taylor expansion
\[\FF\circ\af(t)=\sum_k \frac{F^{(k)}(\bar{\af}(t))}{k!}(\af(t)-\bar{\af}(t)\an)^k,\label{comp}\]
where $\bar{\af}$ is the coordinate of $\af$ onto $\an$. The definition above makes no sense if the product between elements of the regularity structure is not defined. We can also find the general definition in \cite{hairer}. This is not useful here.
The idea of the decomposition \eqref{comp} is to compute a Taylor expansion of $F$ in $\bar{\af}$ the part of $\af$ which is the first approximation of $\Rc\af$.

Here we just prove (what is needed for solving $\eqref{rde}$) that $\FF\circ\af$ lives in the same space as $\af$ and that $\FF$ is Lipschitz in the particular case of modelled distribution of controlled rough paths.

\begin{The}
\label{def_comp}
Let $F\in C^2_b(\RR^d,\mathcal{L}(\RR^n,\RR^d))$. For $\alpha\in (1/3,1/2]$, given a rough path $\rW=(W,\WW)\in \rC^{\alpha}$, the controlled rough path $(y,y')\in \DD^{2\alpha}_W$, for all $\aY\in \rmodel$ defined by $\aY(t)=y_t\an+y'_t\aW$, the map $\FF$ such that
\[\FF\circ\aY(t)\defeq F(y_t)\an+ F'(y_t)y'_t\aW,\label{comp_rough}\]
 is in $\rmodel$. Moreover if $F\in C^3_b$ the function associated $\FF$ is Lipschitz, \emph{i.e.} for all $\aY,\tY\in \rmodel$
 \[ \label{eq_inequality_composition} \rnorme{\FF(\aY)-\FF(\atY)}\leq C\rnorme{\aY-\atY},\]
where $C$ is a constant.
\end{The}

\begin{Rem}
This theorem shows that the space $\rmodel$ is stable by a non linear composition $\FF$, provided that $\FF$ is regular enough.
So with Theorem~\ref{the_sewing_lemma}, we can build the integral
\begin{align*}
I(\FF(\aY))=\int_0^{\cdot} F(y_s)\di{\rW_s}.
\end{align*}
\end{Rem}

\begin{proof}
Firstly, let us show that $\FF$ is a map from $\rmodel$ to $\rmodel$.
A straightforward computation leads us to the two following expressions

$$\norme{\FF(\aY)(t)-\Gamma^r_{t,s}\left(\FF(\aY)(s)\right)}_{0} = \norme{F'(y_t)y'_t-F'(y_s)y'_s},$$ 
$$\norme{\FF(\aY)(t)-\Gamma^r_{t,s}\left(\FF(\aY)(s)\right)}_{\alpha
} =\norme{F(y_t)-F(y_s)-F'(y_s)y'_sW_{s,t}}.$$
Let us denote the left-hand of the first equality $\Delta^0_{s,t}$ and of the second one $\Delta^\alpha_{s,t}$. We obtain
\begin{eqnarray*}
\Delta^0_{s,t} &\leq & \norme{F'(y_t)}\norme{y'_t-y'_s}+\norme{y'_s}\norme{F'(y_t)-F'(y_s)}\\
& \leq & \norme{F'}_{\infty,T}\norme{y'}_{\alpha}\abs{t-s}^\alpha+\norme{y'}_{\infty,T}\norme{F''}_{\infty,T}\norme{Y}_{\alpha,T}\abs{t-s}^\alpha
\end{eqnarray*}
and 
\begin{eqnarray*}
\Delta_{s,t}^\alpha &= &\norme{F(y_t)-F(y_s)-F'(y_s)(y_{s,t}-R^y_{s,t})}\\
&\leq& \norme{F(y_t)-F(y_s)-F'(y_s)y_{s,t}}+\norme{F'(y_s)R^y_{s,t}}\\
&\leq & \frac{1}{2}\norme{F''}_{\infty,T}\norme{y_{s,t}}^2+\norme{F'}_{\infty,T}\norme{R^y}_{2\alpha,T}\abs{t-s}^{2\alpha}\\
&\leq & \frac{1}{2}\norme{F''}_{\infty,T}\norme{y}^2_{\alpha,T}\abs{t-s}^{2\alpha}+\norme{F'}_{\infty,T}\norme{R^y}_{2\alpha,T}\abs{t-s}^{2\alpha}.
\end{eqnarray*}
This proves that $\FF(\aY)\in \rmodel$.

We now prove the inequality \eqref{eq_inequality_composition}. A more general proof can be found in \cite{hairer}. We define $\aZ=\aY-\atY$, which is in $\rmodel$ by linearity. We denote by $\proj$\nomenclature{$\proj$}{The projection onto $\modspace_{<2\alpha}$} the projection onto $\modspace_{<2\alpha}$. Using the integration by parts formula, one can check that
$$\FF(\aY(s))-\FF(\atY(s))=\sum_{k=0}^1\int_0^1F^{(k)}(\ty_s+uz_s)\proj\left[\left[(\ty'_s+uz'_s)\aW\right]^k\aZ(s)\right]\di{u}.$$
Then, we compute the expansion between $s$ and $t$ of $\aDelta(s)\defeq\FF(\aY(s))-\FF(\atY(s))$. We denote $\aA_u(s)\defeq\atY(s)+u\aZ(s)$. When $u$ is fixed, $\aA_u$ is in $\rmodel$. We have

\begin{eqnarray*}
\ga\aDelta(s)&=&\sum_{k=0}^1\int_0^1F^{(k)}(\aA_u(s))\ga\proj\left([A'_u(s)\aW]^k\aZ(s)\right)\di{u}\\
&=&\sum_{k=0}^1\int_0^1F^{(k)}(\aA_u(s))[\ga(A'_u(s)\aW)]^k\ga\aZ(s)\di{u} + \aR(s,t),
\end{eqnarray*}
where $\aR$ is a remainder such that $\norme{\aR(s,t)}_{\beta}\est \abs{t-s}^{2\alpha-\beta}$ for $\beta\in\{0,\alpha\}$. From now, we denote by $\aR$ all the remainder terms which satisfy this property.

We now shift the last expression from $s$ to $t$.
On the one hand
$$\ga^r(A'_u(s)\aW)=\ga^r\aA_u(s)-A_u(s)\an=\aA_u(t)-A_u(s)\an+\aR(s,t).$$
On the other hand
$$\ga^r\aZ(s)=\aZ(t)+\aR(s,t).$$
This yields
$$\ga^r\aDelta(s)=\sum_{k=0}^1\int_0^1F^{(k)}(\aA_u(s))[A'_u(t)\aW+(A_u(t)-A_u(s))\an]^k\aZ(s)\di{u} + \aR(s,t).$$
It remains to shift $F^{(k)}$ from $s$ to $t$. With the classical Taylor expansion formula,
$$F^{(k)}(A_u(s))=\sum_{0\leq l+k\leq1}F^{(k+l)}(A_u(t))(A_u(s)-A_u(t))^l+O(\abs{t-s}^{2\alpha-k\alpha}),$$
because $\norme{A_u(t)-A_u(s)}\leq\abs{t-s}^{\alpha}.$
The bound
$$\norme{[A'_u(t)\aW+(A_u(t)-A_u(s))\an]^k}_\beta\est \abs{t-s}^{k\alpha-\beta}$$
holds. Finally, with the two previous expressions,
\begin{eqnarray*}
\ga^r\aDelta(s)&=&\sum_{0\leq l+k\leq1} F^{(k+l)}(A_u(t))(A_u(s)-A_u(t))^l\\
&\times &[A'_u(t)\aW+(A_u(t)-A_u(s))\an)]^k\aZ(t)+\rnorme{\aZ}O(\abs{t-s}^{2\alpha-\beta})\\
&=& \sum_{0\leq k\leq 1} F^{(k)}(A_u(t))[A'_u(t)\aW]^k\aZ(t)+\rnorme{\aZ}O(\abs{t-s}^{2\alpha-\beta})\\
&=& \aDelta(t)+\rnorme{\aZ}O{\abs{t-s}^{2\alpha-\beta}},
\end{eqnarray*}
which proves the inequality.
\end{proof}

%%%%%%%%%%%%%%%%%%%%%%%%%%%%%%%%%%%%%%%%%%%%%%%%%%%%%%%%%%%%%%%%%%%%
\section{Solving the rough differential equations}
\label{sec_solve}
 Theorem \ref{the_sewing_lemma} combined with  Theorem \ref{def_comp} allow us to solve the rough differential equations in the modelled distribution space $\rmodel$.

\begin{The}
\label{the_abstract_solve}
Given $\xi\in \RR^d$, $F\in C_b^3(\RR^{d},\mathcal{L}(\RR^n,\RR^d))$, a rough path $\rW=(W,\WW)\in \rC^\beta$ with $\beta\in (1/3,1/2)$, there is a unique modelled distribution $\aY\in D^{2\beta}_{M^r}$ such that for all $t\in [0,T],$
\[\aY(t)=\xi \an+L(\FF(\aY))(t),\label{eq_abstract_rde}\]
 where $L$ is defined in Theorem \ref{the_sewing_lemma}.
\end{The}

\begin{proof}
We  prove that the operator $N(\aY)\defeq  \xi\an+L(\FF(\aY))$ where $L$  is defined in  Theorem \ref{the_sewing_lemma}, has a unique fixed point. For this we show that the unit ball of $\rmodel$ is invariant under the action of $N$, and then that $N$ is a strict contraction.

These two properties can be obtained by choosing a wise time  interval $[0,T]$. We take a rough path $\rW=(W,\WW)\in \rC^{\beta}\subset \rC^{\alpha}$ with $1/3<\alpha<\beta<1/2$ and $\aY\in \rmodel$. This trick allows us to have a $T^{\beta-\alpha}$ in our estimates. Thus, with a $T$ small enough we  prove the fixed point property. We start by choosing $T\leq 1$.

According to  Theorem \ref{def_comp} $\FF(\aY)\in \rmodel$, thus  Theorem \ref{the_sewing_lemma} shows that $N(\aY)\in \rmodel$.
If $\aY$ is a fixed point of $N$ then $\aY\in D^{2\beta}_{M^r}$, thanks to the fact that $\rW\in \rC^{\beta}$. Indeed,
$$\norme{\aY(t)-\ga\aY(s)}_\beta=\norme{y_{s,t}}\leq \norme{y'}_{\infty,T}\norme{W}_{2\beta,T}\abs{t-s}^{2\beta}+\norme{R^y}_{2\alpha,T}\abs{t-s}^{2\alpha},$$
and
$$\norme{\aY(t)-\ga\aY(s)}_0=\norme{y_{s,t}-y'_sW_{s,t}}\leq \norme{y'}_{\infty,T}\norme{\WW_{s,t}}+O(\abs{t-s}^{3\alpha}).$$
As a result of the fixed point property $y'=F(y)$. This proves that $\aY\in D^{2\beta}_{M^r}$.

We recall that $\rnorme{\aY}=\sup_{\epsilon\in\{0,\alpha\}}\norme{\aY(0)}_\epsilon + \rsnorme{\aY},$
where $$\rsnorme{\aY}=\sup_{t,s\in [0,T],\epsilon\in \{0,\alpha\}}\frac{\norme{\aY(t)-\ga\aY(s)}_{\epsilon}}{\abs{t-s}^{\epsilon}}.$$
It is more convenient to work with the semi-norm $\rsnorme{\cdot}$, so we define the affine ball unit on $[0,T]$
$$B_T=\{\aY\in \rmodel,~\aY(0)=\xi\an+f(\xi)\aW,~\rsnorme{\aY}\leq 1\}.$$ 
\paragraph{Invariance:} For $\aY\in B_T,$ on has $\rsnorme{\FF(\aY)}\leq \rsnorme{\aY}$ and \begin{align*}
{N(\aY)}= \rsnorme{L(\FF(\aY))}.
\end{align*}
On the on hand, according to the reconstruction map,
\begin{eqnarray*}
&~&\norme{(\aI\FF(\aY))_{s,t}-F(y_s)W_{s,t}}\\
&\leq & \norme{F'(y)y'}_{\infty,T}\norme{\WW}_{2\alpha,T}\abs{t-s}^{2\alpha}+C\rnorme{\FF(\aY)}\abs{t-s}^{3\alpha}\\
&\leq & \rnorme{\FF(\aY)}\norme{\WW}_{2\alpha}\abs{t-s}^{2\alpha}+C\rnorme{\FF(\aY)}\abs{t-s}^{3\alpha}\\
&\leq &\norme{F'}_\infty\left[  \rnorme{(\aY)}\norme{\WW}_{2\beta,T}T^{\beta-\alpha}\abs{t-s}^{2\alpha}+ C\rnorme{\aY}\abs{t-s}^{2\alpha}T^{\alpha}\right],
\end{eqnarray*}
because $\norme{\cdot}_{\beta}\leq \norme{\cdot}_\alpha T^{\beta-\alpha}$.
Using the fact that $T^\alpha\leq T^{\beta-\alpha}$ and that $\aY\in B_T$  we obtain
$$\norme{N(\aY)}_{0}\leq CT^{\beta-\alpha},$$
where $C$ is independent of $\aY$.
On the other hand,
\begin{eqnarray*}
\norme{y_{s,t}}&\leq &\norme{y'}_{\infty,T}\norme{W}_{\alpha,T}\abs{t-s}^{\alpha}+\norme{R^y}_{2\alpha,T}\abs{t-s}^{2\alpha}\\
&\leq & \rnorme{\aY}\norme{W}_{\beta,T} T^{\beta-\alpha}\abs{t-s}^{\alpha}+\norme{R^y}_{2\alpha,T} T^\alpha \abs{t-s}^\alpha\\
&\leq & \rnorme{\aY}\norme{W}_{\beta,T} T^{\beta-\alpha}\abs{t-s}^{\alpha}+\rnorme{\aY} T^\alpha \abs{t-s}^\alpha.
\end{eqnarray*}
Using the last inequality
\begin{eqnarray*}
\norme{F(y)}_{\alpha,T} &\leq & \norme{F'}_\infty\norme{y}_{\alpha,T}\\
&\leq &\rnorme{\aY}\norme{W}_{\beta,T} T^{\beta-\alpha}\abs{t-s}^{\alpha}+\rnorme{\aY} T^\alpha \abs{t-s}^\alpha,
\end{eqnarray*} which leads to $\norme{N(\aY)}_\alpha\leq CT^{\beta-\alpha}$. Finally, we obtain the following estimate $\rsnorme{N(\aY)}\leq CT^{\beta-\alpha}$, where $C$ does not depend on $\aY$. By choosing $T=T_0$ small enough, we show that $N(B_{T_0})\subset B_{T_0}$.
\paragraph{Contraction:} For $\aY,\atY\in \rmodel,$ 
\begin{eqnarray*}
\rsnorme{N(\aY)-N(\atY)}&\leq & \norme{N(\aY)-N(\atY)}_0+\norme{N(\aY)-N(\atY)}_\alpha\\
& \leq &C \rnorme{\FF(Y)-\FF(\atY)}T^{\beta-\alpha}+ \norme{F(y)-F(\ty)}_{\alpha}\\
& \leq &C \rnorme{\aY-\atY}T^{\beta-\alpha}+ \norme{F'}_\infty\norme{y-\ty}_{\alpha},
\end{eqnarray*}
according to \eqref{eq_inequality_composition}. Then it is easy to show that $$\norme{y-\ty}_{\alpha}\leq C T^{\beta-\alpha}\rsnorme{\aY-\atY}.$$
Finally, $\rsnorme{N(\aY)-N(\atY)}\leq CT^{\beta-\alpha}\rsnorme{\aY-\atY}$
where $C$ does not depend on neither $\aY$ nor $\atY$. So with $T$ small enough, $N(B_{T})\subset B_{T}$ and $N$ is a strict contraction. So, there is a unique solution $\aY\in \rmodel$ to \eqref{eq_abstract_rde} on $[0,T]$. As mentioned at the beginning of the proof, $\aY$ is in$ D^{2\beta}_{M^r}$. 
\end{proof}

\begin{Cor}
Given $\xi\in \RR^d$, $F\in C_b^3(\RR^d,\mathcal{L}(\RR^{n},\RR^d))$, a rough path $\rW=(W,\WW)\in \rC^\beta$ with $\beta\in (1/3,1/2)$, there is a unique controlled rough path $(y,y')\in \DD^{2\beta}_W$ such that for all $t\in [0,T]$
\[y(t)=\xi+\int_0^t F(y_u)\di{\rW_u},\]
where the integral has to be understood as the controlled rough path integral (Theorem \ref{classic_rough_int}). 
\end{Cor}

\begin{Rem}
Actually, we can extend this result to $T=+\infty$, because $T$ is chosen uniformly
with respect to parameters of the problem.
\end{Rem}

\begin{proof}
It suffices to project Equation~\eqref{eq_abstract_rde} onto $\an$ and onto $\aW$.
\end{proof}

\section*{Acknowledgments}

I am very grateful to Laure Coutin and Antoine Lejay for their availability, help and their careful rereading.

I deeply thank Peter Friz for suggesting me this topic during my master thesis and for  welcoming me at the Technical University of Berlin for 4 months.

\printnomenclature

\newpage
\bibliographystyle{alpha}
\bibliography{biblio}

\end{document}